\documentclass[final,hidelinks,onefignum,onetabnum]{siamart250211}
\usepackage{amsfonts,amssymb}
\usepackage{epstopdf}
\usepackage{algorithm}
\usepackage{algorithmic}
\usepackage{multirow}
\usepackage{color}
\usepackage{float}
\usepackage{subcaption}
\usepackage{mathrsfs}

\ifpdf
  \DeclareGraphicsExtensions{.eps,.pdf,.png,.jpg}
\else
  \DeclareGraphicsExtensions{.eps}
\fi

\newsiamremark{remark}{Remark}
\usepackage{amsopn}

\headers{Hyperspectral Image Denoising and Destriping}{X. Liu, S. Yu, J. Lu, and X. Chen}

\title{Orthogonal Constrained Minimization with Tensor \(\ell_{2,\MakeLowercase{p}}\) Regularization for HSI Denoising and Destriping
\thanks{Submitted to the editors 3 March 2025; revised version 27 June, 2025.
\funding{This work was partially supported by the Hong Kong Research Council Grant C5036-21E, and N-PolyU507/22,  the PolyU internal grant P0040271, and the National Natural Science Foundation of China under grants U21A20455 and 12326619.}}}

\author{Xiaoxia Liu\thanks{School of Mathematics, South China University of Technology, Guangzhou, Guangdong, 510641, China (\email{xiaoxia\_liu\_math@outlook.com}).} \and
	Shijie Yu\footnotemark[5]~\thanks{Shenzhen Key Laboratory of Advanced Machine Learning and Applications, School of Mathematical Sciences, Shenzhen University, Shenzhen 518060, China.} \and 
	Jian Lu\footnotemark[3]~\thanks{National Center for Applied Mathematics Shenzhen, Shenzhen 518055, China (\email{jianlu@szu.edu.cn}).}\and Xiaojun Chen\thanks{Department of Applied Mathematics, The Hong Kong Polytechnic University, Hong Kong, China (\email{shi-jie.yu@connect.polyu.hk} and \email{xiaojun.chen@polyu.edu.hk}).} 
	 }

\begin{document}
\maketitle

\begin{abstract}
Hyperspectral images~(HSIs) are often contaminated by a mixture of noise such as Gaussian noise, dead lines, stripes, and so on. In this paper, we propose a multi-scale low-rank tensor regularized $\ell_{2,p}$ (MLTL2p) approach for HSI denoising and destriping, which consists of an orthogonal constrained minimization model and an iterative algorithm with convergence guarantees. The model of the proposed MLTL2p approach is built based on a new sparsity-enhanced Multi-scale Low-rank Tensor regularization and a tensor $\ell_{2,p}$ norm with \(p\in (0,1)\). The multi-scale low-rank regularization for HSI denoising utilizes the global and local spectral correlation as well as the spatial nonlocal self-similarity priors of HSIs. The corresponding low-rank constraints are formulated based on independent higher-order singular value decomposition with sparsity enhancement on its core tensor to prompt more low-rankness. The tensor $\ell_{2,p}$ norm for HSI destriping is extended from the matrix $\ell_{2,p}$ norm. A proximal block coordinate descent algorithm is proposed in the MLTL2p approach to solve the resulting nonconvex nonsmooth minimization with orthogonal constraints. We show any accumulation point of the sequence generated by the proposed algorithm converges to a first-order stationary point, which is defined using three equalities of substationarity, symmetry, and feasibility for orthogonal constraints. In the numerical experiments, we compare the proposed method with state-of-the-art methods including a deep learning based method, and test the methods on both simulated and real HSI datasets. Our proposed MLTL2p method demonstrates outperformance in terms of metrics such as mean peak signal-to-noise ratio as well as visual quality.
\end{abstract}



% % -------------------- 1. Introduction -------------------- % %
\section{Introduction}\label{sec:intro}
Hyperspectral images~(HSIs) are collected by hyperspectral sensors across the electromagnetic spectrum. For a three-dimensional (3-D) HSI, the first two dimensions represent spatial information, and the third dimension represents the spectral information of a scene. An illustration of an HSI is shown in Fig. \ref{A Hyperspectral Image}. HSIs are widely used for various applications~\cite{prevost2022hyperspectral,rajwade2013coded,song2019online,zheng2024scale} such as object detection~\cite{yu2024generalized}, material identification~\cite{chan2020two,chang2003hyperspectral,grahn2007techniques}, etc.
\begin{figure}[h]
\centering
\includegraphics[width=8cm]{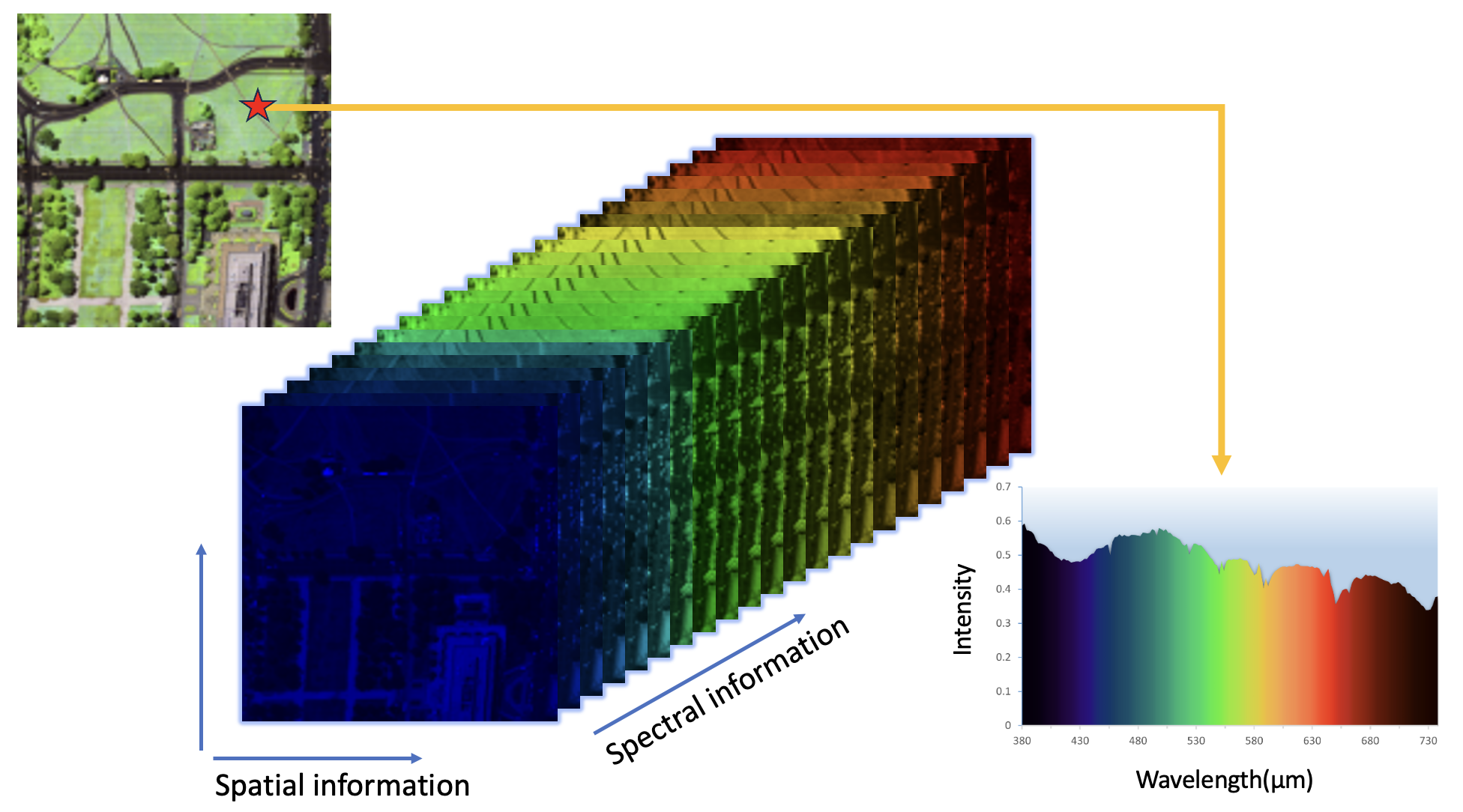}
\caption{An illustrative figure of an HSI.}
\label{A Hyperspectral Image}
\end{figure}

HSIs are often contaminated by additive white Gaussian noise, impulse noise, stripe noise, dead line noise, and so on. Mathematically, a noisy HSI $\mathcal{D}\in \mathbb{R}^{I_{1}\times I_{2}\times I_{3}}$ can be expressed as
\begin{displaymath}\label{1.2}
    \mathcal{D}=\mathcal{L}+\mathcal{S}+\mathcal{N},
\end{displaymath}
where $\mathcal{L}\in \mathbb{R}^{I_{1}\times I_{2}\times I_{3}}$ represents the clean HSI, $\mathcal{S}\in \mathbb{R}^{I_{1}\times I_{2}\times I_{3}}$ represents the sparse noise such as impulse noise, stripe noise and dead line noise, and $\mathcal{N}\in \mathbb{R}^{I_{1}\times I_{2}\times I_{3}}$ represents  additive white Gaussian noise.

To remove  additive white Gaussian noise in HSIs, many methods have been proposed. Conventional 2-D methods~\cite{dabov2007image,elad2006image}, processing HSIs band by band, do not fully utilize the strong correlation between adjacent bands. 3-D methods such as Block-matching 4-D filtering~(BM4D)~\cite{maggioni2012nonlocal}, spectral-spatial adaptive hyperspectral total variation~(SSAHTV) model~\cite{6165657}, and sparse representation methods~\cite{xing2012dictionary,zhao2014hyperspectral} incorporate both spatial and spectral information and outperform conventional methods. However, those methods may need additional preprocessing to remove non-Gaussian noise.

In real-world scenarios, HSIs are often contaminated by more than one type of noise due to atmospheric effects and instrument noise.  Various methods have been proposed to remove the mixed noise, including low-rank matrix based methods, low-rank tensor based methods, and deep learning based methods. Low-rank matrix based methods~\cite{guo2022gaussian,su2023fast,zhang2021double} reshape an HSI into a matrix, e.g., each spectral band of the HSI is reshaped as each row of the matrix, and impose low-rankness on the reshaped HSI. Zhang et al.~\cite{zhang2013hyperspectral} formulated the HSI denoising problem as a low-rank matrix factorization problem and solved it by the ``Go Decomposition" algorithm; Zhang et al.~\cite{zhang2021double} proposed a double low-rank matrix decomposition which utilized the $\ell_1$ norm for the impulse noise and the matrix nuclear norm for stripes, and adopted augmented Lagrangian method (ALM) to solve the model; Yang et al.~\cite{yang2021hyperspectral} also used double low-rankness, but added spatial-spectral total variation~(SSTV) to the model and performed the decomposition on full band blocks~(FBBs) rather than the entire HSI. However, imposing low-rankness on matrices can only exploit redundancy in two dimensions and may potentially lose important structural information due to reshaping.

Low-rank tensor based methods~\cite{de2017low,gao2023tensor,li2023learnable,pan2021orthogonal,zheng2024scale} view HSIs as tensors, and can exploit redundancy in three or more dimensions. Tensor low-rank decompositions can preserve the spatial-spectral correlations and even nonlocal self-similarity if block matching is involved. Wang et al.~\cite{wang2017hyperspectral} used Tucker tensor decomposition and an anisotropic SSTV regularization to characterize the piece-wise smooth structures of the HSI; Chen et al.~~\cite{chen2018destriping} proposed a low-rank tensor decomposition~(LRTD) method, which utilized the higher-order singular value decomposition~(HOSVD) for low-rankness and the matrix $\ell_{2,1}$ norm for characterizing the stripes, and adopted ALM for solving the optimization model. Cao et al.~\cite{cao2019hyperspectral} employed the fact that spectral signatures of pixels lie in a low-dimensional subspace, and proposed a subspace-based nonlocal low-rank and sparse factorization (SNLRSF) method for removing mixed noise in HSI, which conducted nonlocal low-rank factorization via successive singular value decomposition (SVD). Xiong et al. \cite{xiong2019hyperspectral} proposed the LRTFL0 method using low-rank block term decomposition and spectral-spatial $\ell_0$ gradient regularization to achieve gradient smoothness.

For removing Gaussian noise, the existing low-rank based methods mainly utilize global spectral correlation of HSIs, and do not use its nonlocal self-similarity prior~\cite{buades2005non}, except for SNLRSF. This limits the performance of restoring clean HSIs. Also, many of those methods use ALM to solve their low-rank regularization based models without providing any convergence analysis. In their ALM, the augmented Lagrangian subproblem is solved by alternating minimization via one Gauss–Seidel-type iteration, instead of an exact joint minimization. As a result, the convergence is not guaranteed for these nonconvex nonsmooth optimization problems with three or more blocks \cite{chen2016direct}.

Recently, some deep neural networks~\cite{bradbury2016quasi,wei20203,yuan2018hyperspectral} have been proposed to denoise HSIs. Quasi-recurrent neural networks (QRNN)~\cite{bradbury2016quasi} combined recurrent neural networks (RNN) with convolutional neural networks~(CNN); the 3-D version of QRNN (QRNN3D)~\cite{wei20203}, applying quasi-recurrent pooling function along spectrum, can effectively embed both structural spatial-spectral correlation and global correlation along spectrum of HSIs;  CNN  can also be used as a denoiser in a plug-and-play fashion for HSI denoising~\cite{9664348}. However, the performance of these methods may vary significantly depending on factors such as the datasets used, noise characteristics, and the size of the input data.

To remove Gaussian noise and stripes simultaneously, we propose a new optimization model utilizing low-rank tensor regularization and a group sparsity measure, which is formulated as follows
\begin{equation}\label{model:HSI}
\begin{split}
\min_{\substack{\mathcal{S},[\boldsymbol{X}_1],[\boldsymbol{X}_2],\\
[\boldsymbol{X}_3],[\mathcal{G}],\mathcal{L}} } \quad
& \frac{1}{2} \|\mathcal{L}+\mathcal{S}-\mathcal{D}\|_{F}^2+\gamma \|\mathcal{S}\|^{p}_{2,p}+\|[\mathcal{G}]\|_{1,\boldsymbol{w}}\\
&+ \frac{\delta}{2}\|\mathscr{R}(\mathcal{L})- [\mathcal{G}]\times
_{1}[\boldsymbol{X}_1]\times_{2}[\boldsymbol{X}_2]\times_{3}[\boldsymbol{X}_3]\|^{2}_{F}\\
\mbox{s.t. } & [\boldsymbol{X}_i]^{\top}[\boldsymbol{X}_i]=[\boldsymbol{I}_{n_i}], \quad i=1,2,3,
\end{split}
\end{equation}
where
\begin{itemize}
    \item \(\mathcal{S},\mathcal{L}\in \mathbb{R}^{I_{1}\times I_{2}\times I_{3}}\), \([\mathcal{G}]\in\mathbb{R}^{n_1\times n_2\times n_3\times N}\) and \([\boldsymbol{X}_i]\in\mathbb{R}^{m_i\times n_i\times N}\), \(m_i\geq n_i\), \(i=1,2,3\);

    \item \(\|\mathcal{S} \|_{2,p}\) denotes the tensor $\ell_{2,p}$ (quasi-)norm with $p \in (0,1)$ of a third order tensor $\mathcal{S}$ defined by
    \begin{equation} \label{eq:L2pNorm}
\|\mathcal{S}\|_{2,p}=\left(\sum_{i_3=1}^{I_3}\sum_{i_2=1}^{I_2}\|\boldsymbol{s}_{:i_2i_3}\|^{p}_{2}\right)^{\frac{1}{p}}= \left(\sum_{i_3=1}^{I_3}\sum_{i_2=1}^{I_2}\left(\sum_{i_1=1}^{I_1}{s}_{i_1i_2 i_3}^{2}\right)^{\frac{p}{2}}\right)^{\frac{1}{p}},
    \end{equation}
    and the tensor $\ell_{2,p}$ norm is exactly equal to the matrix $\ell_{2,p}$ norm of the unfolding matrix of ${\mathcal{S}}$ along the first dimension, i.e., $\|{\mathcal{S}}\|_{2,p}=\|\boldsymbol{S}_{(1)}\|_{2,p}$;
 
     \item \(\mathscr{R}:\mathbb{R}^{I_{1}\times I_{2}\times I_{3}} \to \mathbb{R}^{m_{1}\times m_{2}\times m_{3}\times N}\) denotes the block extraction operator that extracts blocks and forms a fourth order tensor, for example, the nonlocal similar FBB tensor, which can be regularized to exhibit low-rank properties;
                    
    \item \([\mathcal{G}]\times_{1}[\boldsymbol{X}_1]\times_{2}[\boldsymbol{X}_2]\times_{3}[\boldsymbol{X}_3]\) denotes an \textit{independent 3-D HOSVD} (see the definition in section~\ref{sec:independentHOSVD}) with \([\mathcal{G}]\) being a stack of $N$ independent core tensors and  \([\boldsymbol{X}_i]\) being a stack of $N$ independent mode-$i$ factor matrices, \(i=1,2,3\),  and \([\boldsymbol{X}_i]\) is independently orthogonal, i.e., \([\boldsymbol{X}_i]^{\top}[\boldsymbol{X}_i]=[\boldsymbol{I}_{n_i}]\), with $\boldsymbol{I}_{n_i}$ representing the identity matrix of size $n_i\times n_i$;
    
     \item $\|[\mathcal{G}]\|_{1,\boldsymbol{w}}$ denotes the weighted tensor (component-wise) $\ell_1$ norm for a fourth order tensor \([\mathcal{G}]\) defined by
    \begin{equation}\label{eq:Weightedl1Norm}
    	\|[\mathcal{G}]\|_{1,\boldsymbol{w}}=\sum_{j=1}^N w_j\|[\mathcal{G}]^{(j)}\|_1,
    \end{equation}
    with $\boldsymbol{w}\in \mathbb{R}^N_{++}$ being a positive weight vector, and $[\mathcal{G}]^{(j)}$ being the $j$-th independent core tensor of \([\mathcal{G}]\).
\end{itemize}

Model~\eqref{model:HSI} is a nonconvex nonsmooth minimization problem with orthogonal constraints. In particular, the first term of model~\eqref{model:HSI} is a data fidelity term to remove Gaussian noise, the second term is a group sparsity measure to remove sparse noise with linear structures, and the last two terms are the sparsity-enhanced nonlocal low-rank tensor regularization terms. 

Our main contributions are summarized as follows:

\begin{itemize}
    \item We propose a novel model to remove mixed noise in HSIs using a new sparsity-enhanced low-rank regularization and a tensor $\ell_{2,p}$ norm with \(p\in (0,1)\). For removing Gaussian noise, a tensor formed via the extraction operator $\mathscr{R}$ is regularized on its low-rankness using independent 3-D HOSVD with sparsity enhancement on its core tensors to prompt more low-rankness. The tensor $\ell_{2,p}$ norm for removing the dead lines and stripes is extended from the matrix $\ell_{2,p}$ norm. Some new mathematical results on finding the proximal operator of the tensor $\ell_{2,p}$ norm are also presented. We show that  model~\eqref{model:HSI} has a nonempty and bounded solution set.

    \item We propose a proximal block coordinate descent~(P-BCD) algorithm for solving problem~\eqref{model:HSI}. Each subproblem of the P-BCD algorithm has an exact form solution, which either has a closed-form solution or is easy to compute. We define the first-order stationary point of model~\eqref{model:HSI} using three equalities of substationarity, symmetry, and feasibility for orthogonal constraints. We prove that any accumulation point of the sequence generated by P-BCD algorithm is a first-order stationary point.

    \item We propose a multi-scale low-rank tensor regularized $\ell_{2,p}$ (MLTL2p) approach with two phases for HSI denoising and destriping, which utilizes the global and local spectral correlation as well as the spatial nonlocal self-similarity priors of HSIs.
    We show the MLTL2p approach can outperform other state-of-the-art methods including a deep learning based method on the numerical experiments tested on both simulated and real HSI datasets.
\end{itemize}

The rest of this paper is organized as follows. In section~\ref{sec:notations}, we present some notations and preliminaries for tensors, nonconvex nonsmooth optimization, manifold optimization, and independent 3-D HOSVD. For solving model~\eqref{model:HSI}, we propose a P-BCD method in section~\ref{sec:algorithm} and conduct its convergence analysis in section~\ref{sec:convergence}. Next, we propose the MLTL2p approach for HSI denoising and destriping in section~\ref{sec:application} and conduct experiments on simulated and real HSI data in section~\ref{sec:experiments}. The concluding remarks are given in section~\ref{sec:conclusions}.

\section{Notations and Preliminaries}\label{sec:notations}

 In this section, we first present the notations for tensors, nonconvex nonsmooth optimization, and manifold optimization, then we introduce some notations and preliminaries for the independent 3-D HOSVD.

\subsection{Notations}

	Throughout this paper, tensors are denoted by calligraphic uppercase letters, e.g., $\mathcal{X}$, matrices by boldface uppercase letters, e.g., $\boldsymbol{X}$, vectors by boldface lowercase letters, e.g., $\boldsymbol{x}$, scalars or entries by lowercase letters, e.g., $x$, and operators by curly uppercase letters, e.g., $\mathscr{R}$.

First, we introduce the notations for tensors.
For a third order tensor $\mathcal{X} \in \mathbb{R}^{I_{1}\times I_{2}\times I_{3}}$, we let $x_{i_{1}i_{2}i_{3}}$ denote its $(i_{1},i_{2},i_{3})$-th entry,  let $\boldsymbol{x}_{:i_{2}i_{3}}$ denote its $(i_{2},i_{3})$-th mode-$1$ fiber and  let $\boldsymbol{X}_{::i_{3}}$ denote its $i_{3}$-th frontal slice. And $\mathcal{X} \in \mathbb{R}^{I_{1}\times I_{2}\times I_{3}}_{++}$ means all its entries are positive. The mode-$k$ unfolding of a third order tensor $\mathcal{X}$ is denoted as
$\boldsymbol{X}_{(k)} = {\rm unfold}_{(k)}(\mathcal{X})$, which is the process to linearize all indexes except index $k$. The dimensions of $\boldsymbol{X}_{(k)}$ are $I_{k}\times \prod^{3}_{j=1,j\ne k}I_{j}$. An element $x_{i_{1}i_2i_{3}}$ of $\mathcal{X}$ corresponds to the position of $(i_{k},j)$ in matrix $\boldsymbol{X}_{(k)}$, where $j=1+\sum^{3}_{l=1,l\ne k}(i_{l}-1)\prod^{l-1}_{m=1, m\ne k}I_{m}$. The inverse process of the mode-$k$ unfolding of a tensor $\mathcal{X}$ is denoted by $\mathcal{X} = {\rm fold}_{(k)}(\boldsymbol{X}_{(k)})$.  An illustration of the unfoldings is presented in Fig.~\ref{fig:unfolding}.
\begin{figure}[h]
    \centering
    \includegraphics[width=\linewidth]{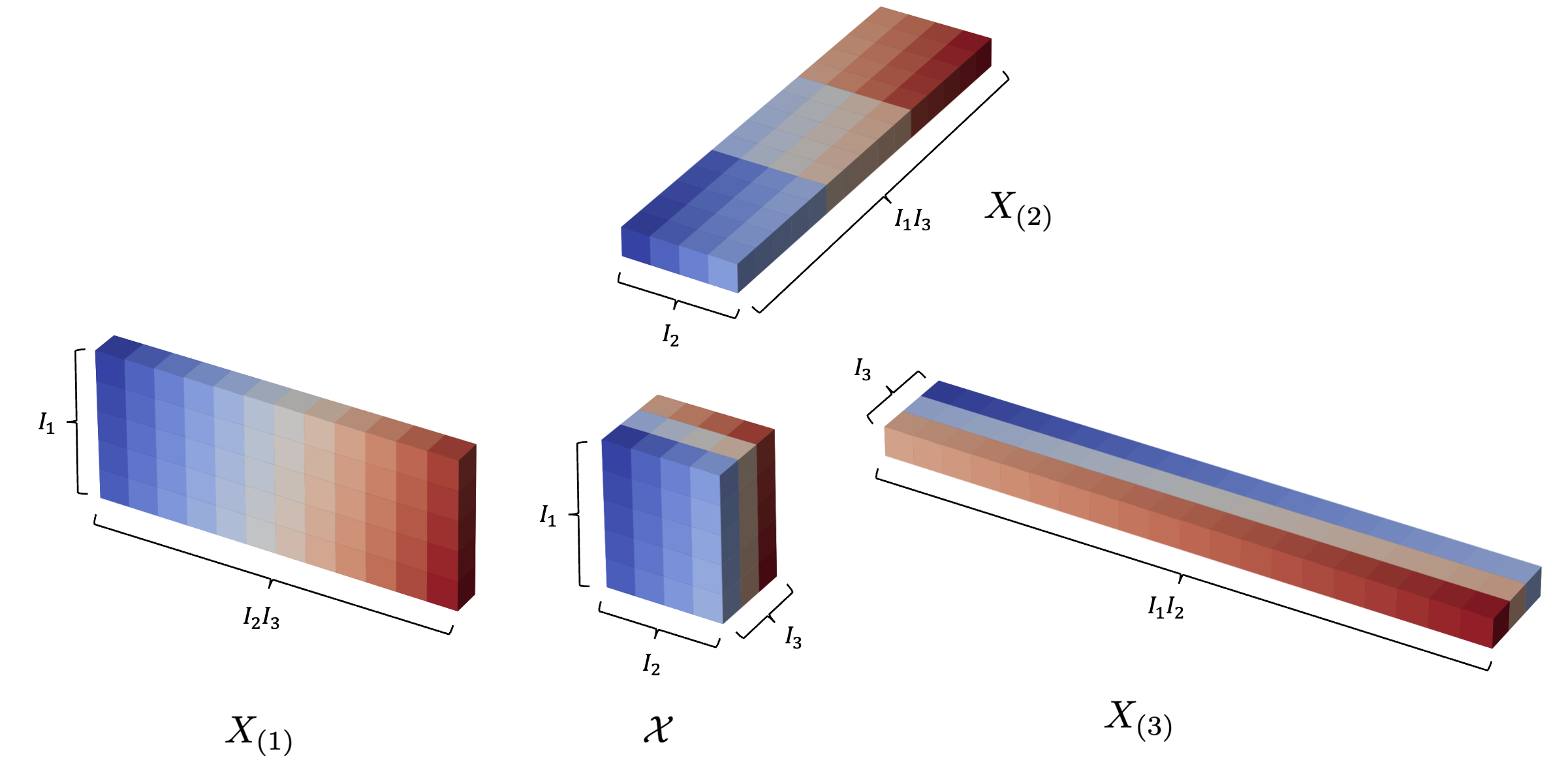}
    \caption{The unfoldings of tensor $\mathcal{X}\in\mathbb{R}^{I_1\times I_2\times I_3}$.}\label{fig:unfolding}
\end{figure}

The (component-wise) $\ell_1$ norm and Frobenius norm of $\mathcal{X}$ are given by
\begin{align*}
    \|\mathcal{X}\|_{1} :&=
\textstyle\sum_{i_{1}=1}^{I_{1}}\sum_{i_{2}=1}^{I_{2}}\sum_{i_{3}=1}^{I_{3}} |x_{i_{1}i_{2}i_{3}}|,
\\
\|\mathcal{X}\|_{F} :&= (\textstyle\sum_{i_{1}=1}^{I_{1}}\sum_{i_{2}=1}^{I_{2}}\sum_{i_{3}=1}^{I_{3}}  |x_{i_{1}i_{2}i_{3}}|^{2})^{\frac{1}{2}}.
\end{align*}

Next, we provide the definitions for (limiting) subdifferentials and proximal operators. Let $f:\mathbb{R}^d\to(-\infty,+\infty]$ be a proper and lower semicontinuous function with a finite lower bound. The (limiting) subdifferential of $f$ at $\boldsymbol{x}\in \operatorname{dom}f:=\{\boldsymbol{x}\in\mathbb{R}^{d}|f(\boldsymbol{x})<\infty\}$, denoted by $\partial f(\boldsymbol{x})$, is defined as
\begin{align*}
        \partial f(\boldsymbol{x})\!:=\{\boldsymbol{u} \in\mathbb{R}^{d}|&\;\exists \boldsymbol{x}^{k}\to \boldsymbol{x}, f(\boldsymbol{x}^k)\to f(\boldsymbol{x})\\
        &  \text{ and } \boldsymbol{u}^{k}\to \boldsymbol{u} \text{ with } \boldsymbol{u}^{k}\in \hat{\partial}f(\boldsymbol{x}^{k}) \mbox{ as }k\to \infty\},
\end{align*}
where $\hat{\partial}f(\boldsymbol{x})$ denotes the Fr\'echet subdifferential of $f$ at $\boldsymbol{x} \in \operatorname{dom}f$, which is the set
of all $\boldsymbol{u} \in \mathbb{R}^{d}$ satisfying
\begin{equation}
    \begin{split}
      \underset{\boldsymbol{y}\ne \boldsymbol{x},\boldsymbol{y}\rightarrow \boldsymbol{x}}{\lim\inf}\frac{f(\boldsymbol{y})-f(\boldsymbol{x})-\langle \boldsymbol{u},\boldsymbol{y}-\boldsymbol{x}\rangle}{\|\boldsymbol{y}-\boldsymbol{x}\|}\ge 0.
    \end{split}
\end{equation}
It follows from \cite{mordukhovich2006variational} that $\{\boldsymbol{u} \in\mathbb{R}^{d}|\exists \boldsymbol{x}^{k}\to \boldsymbol{x}, f(\boldsymbol{x}^k)\to f(\boldsymbol{x})  \text{ and } \boldsymbol{u}^{k}\to \boldsymbol{u} \text{ with } \boldsymbol{u}^{k}\in \partial f(\boldsymbol{x}^{k}) \mbox{ as }k\to \infty\} \subseteq \partial f(\boldsymbol{x})$. Next, the proximal operator of $f$ with parameter $\lambda>0$ evaluated at $\boldsymbol{x}\in \mathbb{R}^d$, denoted as $\text{prox}_{\lambda f}(\boldsymbol{x})$, is defined as
\[
\text{prox}_{\lambda f}(\boldsymbol{x}): = \underset{\boldsymbol{u}\in \mathbb{R}^d}{\arg\!\min} \left[f(\boldsymbol{u}) + \frac{1}{2\lambda} \|\boldsymbol{u} - \boldsymbol{x}\|_2^2\right].
\]
Note that $\text{prox}_{\lambda f}$ is a single-valued map, if $f$ is a convex function. However, when $f$ is nonconvex,  $\text{prox}_{\lambda f}(\boldsymbol{x})$ may have multiple points. 

Also, we set $\mathbb{S}_{m,n}:=\{\boldsymbol{X}\in\mathbb{R}^{m\times n}|\boldsymbol{X}^{\top}\boldsymbol{X}=\boldsymbol{I}_{n}\}$ as the Stiefel manifold with $ m\geq n$
and set $\mathcal{T}_{\boldsymbol{X}} \mathbb{S}_{m,n}:=\{\boldsymbol{Y}\in\mathbb{R}^{m\times n}| \boldsymbol{Y}^{\top}\boldsymbol{X}+\boldsymbol{X}^{\top}\boldsymbol{Y}=\boldsymbol{0}\}$
as the tangent space of Stiefel manifold at $\boldsymbol{X}\in\mathbb{R}^{m\times n}$. We also set the Riemannian metric on Stiefel manifold as the metric induced from the Euclidean inner product. Then according to \cite{2008Optimization}, the Riemannian gradient of  a smooth function $f$ at $\boldsymbol{X}$ is given by
\begin{displaymath}
    \operatorname{grad} f(\boldsymbol{X}) :=\operatorname{Proj}_{\mathcal{T}_{\boldsymbol{X}} \mathbb{S}_{m,n}}(\nabla f(\boldsymbol{X})),
\end{displaymath}
where $\operatorname{Proj}_{\mathcal{T}_{\boldsymbol{X}} \mathbb{S}_{m,n}}(\boldsymbol{Y}):= (\boldsymbol{I}_m  -  \boldsymbol{XX}^{\top} )\boldsymbol{Y} + \frac{1}{2} \boldsymbol{X}(\boldsymbol{X}^{\top} \boldsymbol{Y}  -  \boldsymbol{Y}^{\top} \boldsymbol{X})$.

\subsection{Independent 3-D HOSVD}\label{sec:independentHOSVD}

We introduce the definition of a 3-D HOSVD and then define an independent 3-D HOSVD using the notation $[\,\cdot\, ]$. For a third order tensor \(\mathcal{Y}\in\mathbb{R}^{m_{1}\times m_{2}\times m_{3}}\), the (truncated) 3-D HOSVD of \(\mathcal{Y}\) is to approximate $\mathcal{Y} $ in the following form
\begin{equation}\label{eq:3DHOSVD}
    \mathcal{Y} \approx\mathcal{G}\times_1 \boldsymbol{X}_1\times_2 \boldsymbol{X}_2\times_3 \boldsymbol{X}_3,
\end{equation}
where $\mathcal{G}\in \mathbb{R}^{n_1\times n_2\times n_3}$ is the core tensor, and  $\boldsymbol{X}_i\in\mathbb{R}^{m_{i}\times n_{i}}$ is the $i$-th factor matrix such that $\boldsymbol{X}_i^{\top}\boldsymbol{X}_i=\mathbf{I}_{n_i}$. Note that $ m_i\geq n_i$ and $\boldsymbol{X}_i$ belongs to a Stiefel manifold, i.e.,  $\boldsymbol{X}_i\in \mathbb{S}_{m_{i},n_{i}}$. By imposing orthogonality on the factor matrices, the decomposition in \eqref{eq:3DHOSVD} can inherit some nice properties from the matrix SVD. For example, the core tensor can have the all-orthogonality and the ordering property~\cite{Chen2009}.

When a fourth order tensor has little correlation across the last mode, we view the fourth order tensor as a stack of independent third order tensors. For example, for each nonlocal self-similar group after the block extraction operator $\mathscr{R}$ corresponding to different image pattern, different nonlocal self-similar groups with low correction/similarity can be processed independently. Using the notation $[\,\cdot\, ]$, we denote such a fourth order tensor as \([\mathcal{Y}]\in\mathbb{R}^{m_{1}\times m_{2}\times m_{3}\times N}\) and its $j$-th third order tensor as $[\mathcal{Y}]^{(j)}\in\mathbb{R}^{m_{1}\times m_{2}\times m_{3}}$, $j=1,2,\dots, N$, where $N$ denotes the number of independent third order tensors in the stack. Similarly, a stack of independent matrices is denoted as \([\boldsymbol{X}]\in\mathbb{R}^{m\times n\times N}\),  and its $j$-the matrix is $[\boldsymbol{X}]^{(j)}$. An illustration of the stacking processes of independent tensors is presented in Fig.~\ref{fig:independent}.
\begin{figure}[h]
	\centering
		\includegraphics[height=.25\linewidth]{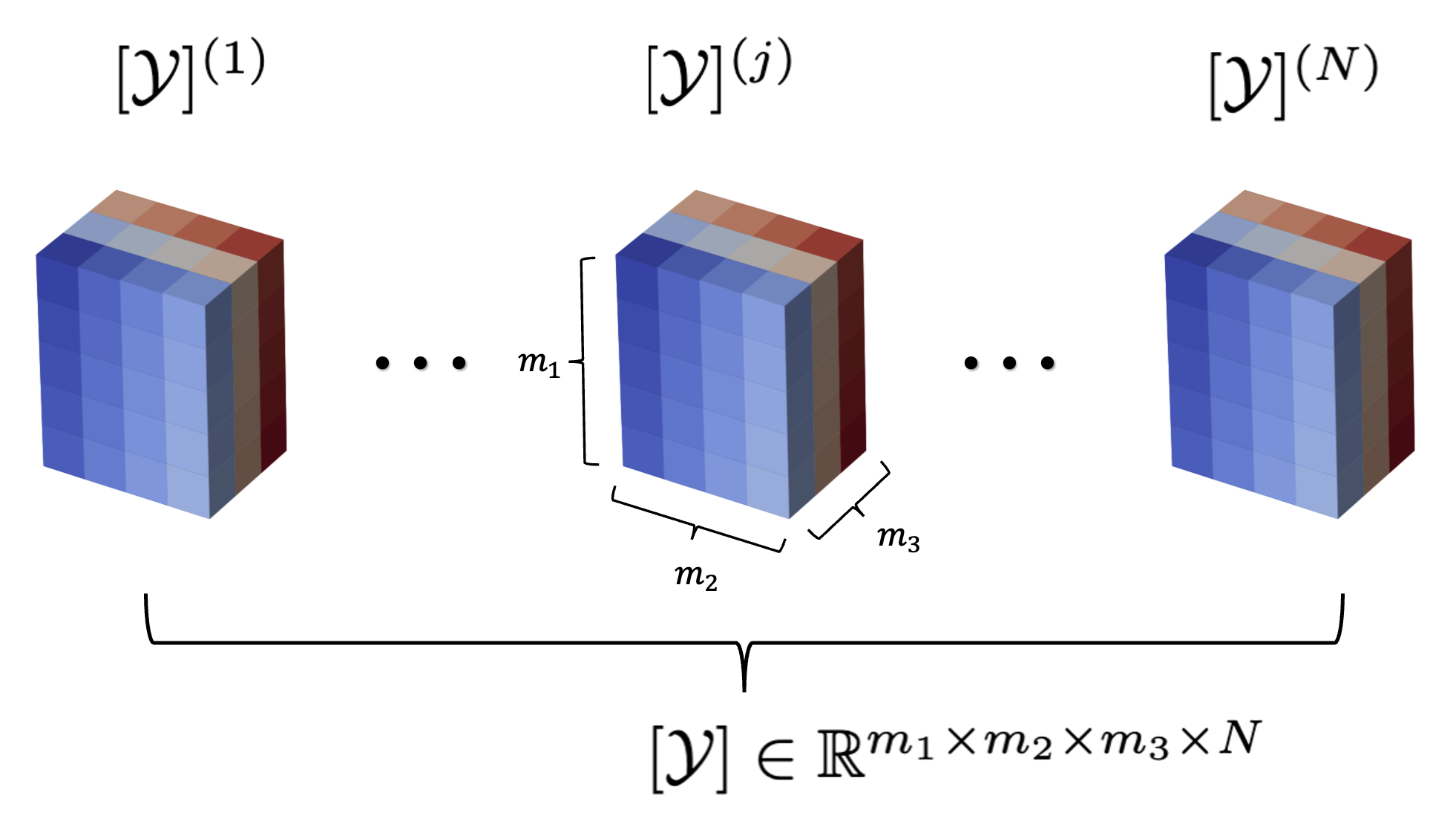}\quad\quad\quad \includegraphics[height=.25\linewidth]{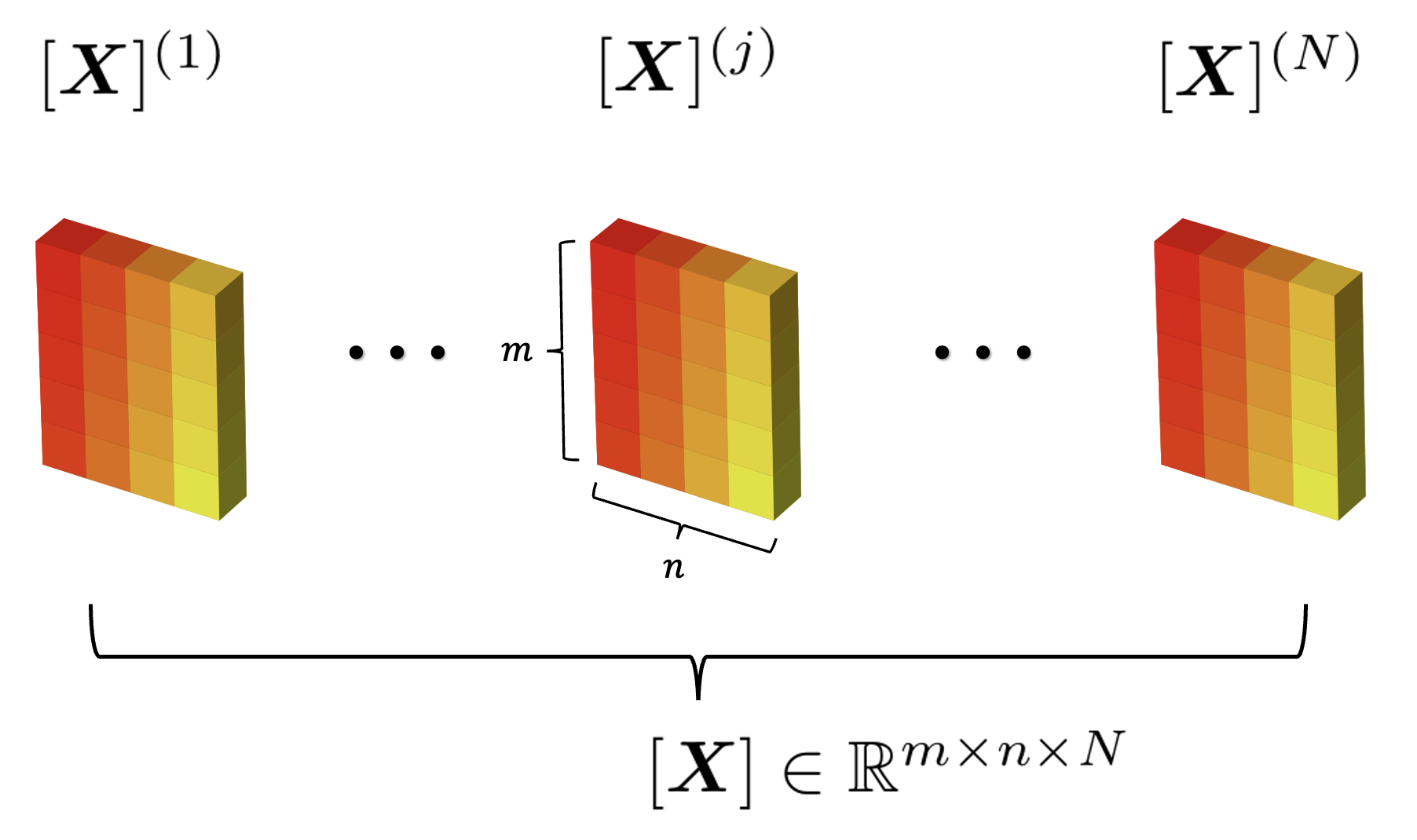}
	\caption{A stack of independent tensors.}
	\label{fig:independent}
\end{figure}

Next, we say \([\boldsymbol{X}]\) is independently orthogonal if \([\boldsymbol{X}]^{\top}[\boldsymbol{X}]=[\mathbf{I}_{n}],\;\) meaning
 \(([\boldsymbol{X}]^{(j)})^{\top}[\boldsymbol{X}]^{(j)}=\mathbf{I}_{n}\), or equivalently  $[\boldsymbol{X}]\in [\mathbb{S}_{m,n}]$, meaning $[\boldsymbol{X}]^{(j)}\in \mathbb{S}_{m,n}$. Then we define an independent (truncated) 3-D HOSVD of \( [\mathcal{Y}]\) as
\begin{equation}\label{eq:Indep3DHOSVD}
    [\mathcal{Y}] \approx[\mathcal{G}]\times_1 [\boldsymbol{X}_1]\times_2 [\boldsymbol{X}_2]\times_3 [\boldsymbol{X}_3],
\end{equation}
where  \([\mathcal{G}]\in\mathbb{R}^{n_1\times n_2\times n_3\times N}\),  \([\boldsymbol{X}_i]\in\mathbb{R}^{m_i\times n_i\times N}\), and
\([\mathcal{Y}]^{(j)} \approx[\mathcal{G}]^{(j)}\times_1 [\boldsymbol{X}_1]^{(j)}\times_2 [\boldsymbol{X}_2]^{(j)}\times_3 [\boldsymbol{X}_3]^{(j)}\) with \(([\boldsymbol{X}_i]^{(j)})^{\top}[\boldsymbol{X}_i]^{(j)}=\mathbf{I}_{n_i}\), \(i=1,2,3\), \(j=1,2,\dots,N\). Similarly, we extend the notation of $[\,\cdot\, ]$ to other operations acting on independent tensors. That is, performing an operation on independent tensors means performing the operation on each lower order tensor independently. For example, performing the mode-$k$ unfolding on \([\mathcal{G}]\), denoted as  \([\mathcal{G}]_{[(k)]}\), means independently performing \([\mathcal{G}]^{(j)}_{(k)}\) for each $j$, which is the mode-$k$ unfolding on \([\mathcal{G}]^{(j)}\).

The core tensor $\mathcal{G}$ in the 3-D HOSVD in \eqref{eq:3DHOSVD}, as well as the stack of independent core tensors $[\mathcal{G}]$ in \eqref{eq:Indep3DHOSVD}, play important roles in characterizing the low-rankness of the tensors $\mathcal{Y}$ and  $[\mathcal{Y}]$, respectively. Specifically, the core tensors and their entries reflect the level of interaction among the different components~\cite{kolda2009tensor}. Therefore, low-rankness can be promoted by imposing the sparsity on the core tensors.

% % -------------------- 3. Algorithm Section -------------------- % %

\section{Proximal Block Coordinate Descent Algorithm}\label{sec:algorithm}

The proposed low-rank tensor model \eqref{model:HSI} is a nonconvex and nonsmooth optimization problem over Stiefel manifolds.  In this section, we propose a P-BCD algorithm for solving model \eqref{model:HSI}.

Let $\Phi$ denote the objective function of model \eqref{model:HSI} defined by
\begin{align*}
    \Phi(\mathcal{S},[\boldsymbol{X}_1],[\boldsymbol{X}_2],[\boldsymbol{X}_3],[\mathcal{G}],\mathcal{L}):= & \frac{1}{2} \|\mathcal{L}+\mathcal{S}-\mathcal{D}\|_{F}^2+\gamma \varphi\left(\mathcal{S}\right)+\psi([\mathcal{G}])\\
    &+ \delta H([\boldsymbol{X}_1],[\boldsymbol{X}_2],[\boldsymbol{X}_3],[\mathcal{G}],\mathcal{L}),
\end{align*}
where  $\varphi(\mathcal{S})=\|\mathcal{S}\|_{2,p}^p$, $\psi([\mathcal{G}])=\|[\mathcal{G}] \|_{1,\boldsymbol{w}}$, and
\begin{equation}\label{Eq:DefH}
 H([\boldsymbol{X}_1],[\boldsymbol{X}_2],[\boldsymbol{X}_3],[\mathcal{G}],\mathcal{L}):=\frac{1}{2}\|\mathscr{R}(\mathcal{L})- [\mathcal{G}]\times
_{1}[\boldsymbol{X}_1]\times_{2}[\boldsymbol{X}_2]\times_{3}[\boldsymbol{X}_3]\|^{2}_{F}.
\end{equation}
Then the P-BCD algorithm is summarized as follows:
\begin{align*}
 \mathcal{S}^{k+1}&\in \underset{\mathcal{S}}{\arg\!\min}\; \Phi(\mathcal{S},[\boldsymbol{X}_1^k],[\boldsymbol{X}_2^k],[\boldsymbol{X}_3^k],[\mathcal{G}^k],\mathcal{L}^k)+\frac{\alpha_{\mathcal{S}}}{2}\|\mathcal{S}-\mathcal{S}^{k}\|^{2}_{F},\\
[\boldsymbol{X}_1^{k+1}]&\in \underset{[\boldsymbol{X}_1]\in[\mathbb{S}_{m_1,n_1}]}{\arg\!\min}\!\Phi(\mathcal{S}^{k+1},[\boldsymbol{X}_1],[\boldsymbol{X}_2^k],[\boldsymbol{X}_3^k],[\mathcal{G}^k],\mathcal{L}^k)+\frac{\alpha_{\boldsymbol{X}}}{2}\|[\boldsymbol{X}_1]-[\boldsymbol{X}_1^{k}]\|^{2}_{F},\\
 [\boldsymbol{X}_2^{k+1}]&\in \underset{[\boldsymbol{X}_2]\in[\mathbb{S}_{m_2,n_2}]}{\arg\!\min}\!\Phi(\mathcal{S}^{k+1},[\boldsymbol{X}_1^{k+1}],[\boldsymbol{X}_2],[\boldsymbol{X}_3^k],[\mathcal{G}^k],\mathcal{L}^k)+\frac{\alpha_{\boldsymbol{X}}}{2}\|[\boldsymbol{X}_2]-[\boldsymbol{X}_2^{k}]\|^{2}_{F},\\
 [\boldsymbol{X}_3^{k+1}]&\in \underset{[\boldsymbol{X}_3]\in[\mathbb{S}_{m_3,n_3}]}{\arg\!\min}\!\Phi(\mathcal{S}^{k+1},[\boldsymbol{X}_1^{k+1}],[\boldsymbol{X}_2^{k+1}],[\boldsymbol{X}_3],[\mathcal{G}^k],\mathcal{L}^k)+\frac{\alpha_{\boldsymbol{X}}}{2}\|[\boldsymbol{X}_3]-[\boldsymbol{X}_3^{k}]\|^{2}_{F},\\
 [\mathcal{G}^{k+1}]&=  \underset{[\mathcal{G}]}{\arg\!\min}\;\Phi(\mathcal{S}^{k+1},[\boldsymbol{X}_1^{k+1}],[\boldsymbol{X}_2^{k+1}],[\boldsymbol{X}_3^{k+1}],[\mathcal{G}],\mathcal{L}^k)+\frac{\alpha_{\mathcal{G}}}{2}\|[\mathcal{G}]-[\mathcal{G}^{k}]\|^{2}_{F},\\
 \mathcal{L}^{k+1}&=  \underset{\mathcal{L}}{\arg\!\min}\;\Phi(\mathcal{S}^{k+1},[\boldsymbol{X}_1^{k+1}],[\boldsymbol{X}_2^{k+1}],[\boldsymbol{X}_3^{k+1}],[\mathcal{G}^{k+1}],\mathcal{L}),
\end{align*}
where parameters \(\alpha_{\mathcal{S}},\alpha_{\boldsymbol{X}},\alpha_{\mathcal{G}}>0\).

In the following, we present the details for computing each update. We will conduct a convergence analysis for the proposed P-BCD algorithm in the next section.

 \subsection{The update of $\mathcal{S}$}

We update $\mathcal{S}^{k+1}$ using the proximal operator of $\varphi$ as follows
\begin{align}
    \mathcal{S}^{k+1}\in &\underset{\mathcal{S}}{\arg\!\min}\; \tilde{\gamma}\varphi\left(\mathcal{S}\right)+ \frac{1}{2}\left\| \mathcal{S}-(\mathcal{S}^k-\tilde{\alpha}_{\mathcal{S}}\left(\mathcal{S}^k+\mathcal{L}^k- \mathcal{D})\right)\right\|_{F}^2,\nonumber\\
    = &\operatorname{prox}_{\tilde{\gamma}\varphi}\left( \mathcal{S}^k-\tilde{\alpha}_{\mathcal{S}}\left(\mathcal{S}^k+\mathcal{L}^k- \mathcal{D}\right)\right),\label{eq:Supdate}
    \end{align}
where $\tilde{\gamma}=\frac{\gamma}{1+\alpha_{\mathcal{S}}}$ and $\tilde{\alpha}_{\mathcal{S}}=\frac{1}{1+\alpha_{\mathcal{S}}}$.

\subsection{The update of $[\boldsymbol{X}_i]$}

Before we solve the optimization subproblem in terms of $[\boldsymbol{X}_i]$ over independent Stiefel manifolds, we can rewrite its objective function using the following useful fact for unfolding of tensors
\begin{displaymath}
	\mathcal{Y}=\mathcal{G}\times_{i} \boldsymbol{X} \mbox{ if and only if }\boldsymbol{Y}_{(i)}=\boldsymbol{X}\boldsymbol{G}_{(i)}.
\end{displaymath}
Then by applying $\boldsymbol{X}\in\mathbb{S}_{m,n}$,  the Frobenious norm of tensors can be rewritten into the Frobenious norm of matrices
\begin{align}
   \|\mathcal{G}\times_{i} \boldsymbol{X}-\mathcal{L}\|_F^2 =& \|\boldsymbol{X}\boldsymbol{G}_{(i)}-\boldsymbol{L}_{(i)}\|_F^2\nonumber\\
    =&\|\boldsymbol{G}_{(i)}\|_F^2-2\langle \boldsymbol{X}\boldsymbol{G}_{(i)},\boldsymbol{L}_{(i)}\rangle+\|\boldsymbol{L}_{(i)}\|_F^2,\label{eq:XG}\end{align}
where $\boldsymbol{G}_{(i)}$ and $\boldsymbol{L}_{(i)}$ denote the mode-$i$ unfolding of $\mathcal{G}$ and $\mathcal{L}$, respectively. Since $\langle \boldsymbol{X}\boldsymbol{G}_{(i)},\boldsymbol{L}_{(i)}\rangle=\langle \boldsymbol{X},\boldsymbol{L}_{(i)}\boldsymbol{G}_{(i)}^{\top}\rangle$ and $\|\boldsymbol{X}\|_F^2=n$, minimizing  $\|\mathcal{G}\times_{i} \boldsymbol{X}-\mathcal{L}\|_F^2$ over $X$ on the Stiefel manifold is equivalent to minimizing $\|\boldsymbol{X}-\boldsymbol{L}_{(i)}\boldsymbol{G}_{(i)}^{\top}\|_F^2$ over the Stiefel manifold. Hence, $[\boldsymbol{X}_i^{k+1}]$ can be computed via the  projection of unfolding matrices onto the Stiefel manifolds independently as follows
\begin{align}\label{eq:Xupdate}
[\boldsymbol{X}_i^{k+1}]\in\operatorname{Proj}_{[\mathbb{S}_{m_{i},n_{i}}]}\left( [\boldsymbol{X}_i^k]- \tilde{\alpha}_{\boldsymbol{X}}\left( [\boldsymbol{X}_i^k]-[\boldsymbol{P}_{i}^k][\boldsymbol{Q}_{i}^k]^{\top}\right)\right),
\end{align}
where
$[\boldsymbol{P}_{i}^k]=(\mathscr{R}(\mathcal{L}^{k}))_{[(i)]}$, $[\boldsymbol{Q}_{1}^k]= \left([\mathcal{G}^k]\times_2 [\boldsymbol{X}^k_2]\times_3 [\boldsymbol{X}^k_3]\right)_{[(1)]}$, $[\boldsymbol{Q}_{2}^k]= \big([\mathcal{G}^k]\times_1 [\boldsymbol{X}_1^{k+1}]$ $\times_3 [\boldsymbol{X}^k_3]\big)_{[(2)]}$, $[\boldsymbol{Q}_{3}^k]= \left([\mathcal{G}^k]\times_1 [\boldsymbol{X}^{k+1}_1]\times_2 [\boldsymbol{X}^{k+1}_2]\right)_{[(3)]}$,
and parameter $\tilde{\alpha}_{\boldsymbol{X}}=\frac{\delta}{\delta+\alpha_{\boldsymbol{X}}}$.

In the following, we present a lemma for finding the projection onto a Stiefel manifold, which is given in Theorem~4.1 in~\cite{high89n} and proved in~\cite{2012Projection}.
\begin{lemma}[\cite{high89n,2012Projection}]\label{Thm:Stiefel} Given $\boldsymbol{A}\in \mathbb{R}^{m\times n}$, $m\geq n$, consider the following Stiefel manifold projection problem
\begin{equation}\label{min:stiefelproj}
\begin{split}
&\min_{\boldsymbol{X}\in \mathbb{R}^{m\times n}} \|\boldsymbol{X}-\boldsymbol{A}\|^{2}_{F}\\
&{\rm s.t. } \quad \boldsymbol{X}^{\top}\boldsymbol{X}=\boldsymbol{I}_n.
\end{split}
\end{equation}
Then the set of optimal solutions of problem~\eqref{min:stiefelproj}, denoted as $\Omega^*(\boldsymbol{A})$, is given by
\begin{align*}
    \Omega^*(\boldsymbol{A})=\{\boldsymbol{U}\boldsymbol{V}^{\top}| &\boldsymbol{A}=\boldsymbol{U}\boldsymbol{\Sigma} \boldsymbol{V}^{\top}, \boldsymbol{U}\in \mathbb{R}^{m\times n}, \boldsymbol{\Sigma}\in\mathbb{R}^{n\times n}, \boldsymbol{V}\in\mathbb{R}^{n\times n}  \\&\text{ such that }
    \boldsymbol{U}^{\top}\boldsymbol{U}=\boldsymbol{V}^{\top}\boldsymbol{V}=\boldsymbol{I}_n \text{ and }\boldsymbol{\Sigma}=\operatorname{Diag}(\boldsymbol{\sigma}(A))\},
\end{align*}
where $\boldsymbol{U}\boldsymbol{\Sigma} \boldsymbol{V}^{\top}$ is a reduced SVD of $\boldsymbol{A}$ and $\boldsymbol{\sigma}(\boldsymbol{A})\in \mathbb{R}^{n}$ is a vector of all the singular values of $\boldsymbol{A}$.
In particular, if $\boldsymbol{A}$ is of full column rank $n$, then $\Omega^*(\boldsymbol{A})$ is a singleton.
\end{lemma}
According to Lemma \ref{Thm:Stiefel}, problem~\eqref{min:stiefelproj} has a closed form solution, even though it may have multiple solutions when the given matrix does not have full column rank. Hence, if $[\boldsymbol{U}][\boldsymbol{\Sigma}] [\boldsymbol{V}]^{\top}$ is an independent reduced SVD of $[\boldsymbol{A}]\in\mathbb{R}^{m\times n\times N}$, then $[\boldsymbol{U}][\boldsymbol{V}]^{\top}\in \operatorname{Proj}_{[\mathbb{S}_{m,n}]}([\boldsymbol{A}])$.

\subsection{The update of $[\mathcal{G}]$}

The subproblem for updating $[\mathcal{G}]$ can be reformulated by using the following property that for any $\boldsymbol{X}\in \mathbb{S}_{m,n}$, $m\geq n$,
\begin{align*}
   &\|\mathcal{G}\times_{i} \boldsymbol{X}-\mathcal{L}\|_F^2\\
   =&\|\boldsymbol{G}_{(i)}-\boldsymbol{X}^{\top}L_{(i)}\|_F^2-\|\boldsymbol{X}^{\top}\boldsymbol{L}_{(i)}\|_F^2+\|\boldsymbol{L}_{(i)}\|_F^2\\
   =&\|\mathcal{G}-\mathcal{L}\times_i \boldsymbol{X}^{\top}\|_F^2-\|\mathcal{L}\times_i \boldsymbol{X}^{\top}\|_F^2+\|\mathcal{L}\|_F^2,
    \end{align*}
which is derived from \eqref{eq:XG} and the constraint that $\boldsymbol{X}^{\top}\boldsymbol{X}=\boldsymbol{I}_n$. Then  $[\mathcal{G}^{k+1}]$ can be computed by
\begin{align}
    [\mathcal{G}^{k+1}]=& \underset{[\mathcal{G}]}{\arg\!\min}\;\psi([\mathcal{G}])+ \frac{\delta}{2}\| [\mathcal{G}]-[\mathcal{O}^{k+1}]\|^{2}_{F}+\frac{\alpha_{\mathcal{G}}}{2}\|[\mathcal{G}]-[\mathcal{G}^{k}]\|^{2}_{F},\nonumber\\
     =&  \underset{[\mathcal{G}]}{\arg\!\min}\;\tilde{w}_{\mathcal{G}}\psi([\mathcal{G}])+ \frac{1}{2}\| [\mathcal{G}]-\left([\mathcal{G}^k]-\tilde{\alpha}_{\mathcal{G}}([\mathcal{G}^{k}]-[\mathcal{O}^{k+1}])\right)\|^{2}_{F},\nonumber\\ =&\operatorname{prox}_{\tilde{w}_{\mathcal{G}}\psi}\left([\mathcal{G}^{k}] -\tilde{\alpha}_{\mathcal{G}}([\mathcal{G}^{k}] -[\mathcal{O}^{k+1}])\right),\label{eq:Gupdate}
\end{align}
where $[\mathcal{O}^{k+1}]=\mathscr{R}(\mathcal{L}^{k})\times_1 [\boldsymbol{X}_1^{k+1}]^{\top}\times_2 [\boldsymbol{X}_2^{k+1}]^{\top}\times_3 [\boldsymbol{X}_3^{k+1}]^{\top}$, $\tilde{w}_{\mathcal{G}}=\frac{1}{\delta+\alpha_{\mathcal{G}}}$, and $\tilde{\alpha}_{\mathcal{G}}=\frac{\delta}{\delta+\alpha_{\mathcal{G}}}$.

\subsection{The update of $\mathcal{L}$} 
Before we compute $\mathcal{L}^{k+1}$, we define the transpose of \(\mathscr{R}\) by \(\mathscr{R}^{\top}:\) \(\mathbb{R}^{m_{1}\times m_{2}\times m_{3}\times N} \to \mathbb{R}^{I_{1}\times I_{2}\times I_{3}}\) satisfying \(\langle \mathscr{R}(\mathcal{L}),[\mathcal{Y}]\rangle=\langle \mathcal{L},\mathscr{R}^{\top}([\mathcal{Y}])\rangle\) for any \(\mathcal{L}\in \mathbb{R}^{I_{1}\times I_{2}\times I_{3}}\), and $[\mathcal{Y}]\in\mathbb{R}^{m_{1}\times m_{2}\times m_{3}\times N}$. Then we have
\begin{displaymath}
    \|\mathscr{R}(\mathcal{L})\|_F^2=\langle \mathcal{L},\mathscr{R}^{\top}\mathscr{R}(\mathcal{L})\rangle_F=\langle \mathcal{L},\mathcal{W}_{\mathscr{R}}\odot \mathcal{L}\rangle_F,%=\|\sqrt{\mathcal{W}_{\mathscr{R}}}\odot \mathcal{L}\|_{F}^2,
\end{displaymath}
where \(\mathcal{W}_{\mathscr{R}}\in \mathbb{R}_{++}^{I_{1}\times I_{2}\times I_{3}}\) denotes a weight tensor such that the component-wise multiplication with \(\mathcal{W}_{\mathscr{R}}\) is an equivalent operation of $\mathscr{R}^{\top}\mathscr{R}$, i.e., $\mathcal{W}_{\mathscr{R}}\odot \operatorname{Id}=\mathscr{R}^{\top}\mathscr{R}$, $\odot$ denotes the component-wise multiplication, and $\operatorname{Id}$ denotes the identity mapping on $\mathbb{R}^{I_{1}\times I_{2}\times I_{3}}$. Then $\mathcal{L}^{k+1}$ can be computed in a unique closed form  as follows
\begin{align}\label{eq:Lupdate}
    \mathcal{L}^{k+1}= &(\delta\mathcal{W}_\mathscr{R}+\mathcal{I})^{-1}\odot\left(\delta\mathscr{R}^{\top}([\mathcal{Y}^{k+1}]) + \mathcal{D}-\mathcal{S}^{k+1}\right),
\end{align}
where  $[\mathcal{Y}^{k+1}]$ denotes the approximated low-rank tensor
\begin{equation}\label{eq:Yupdate}
	[\mathcal{Y}^{k+1}]=[\mathcal{G}^{k+1}]\times_1 [ \boldsymbol{X}_1^{k+1}]\times_2 [ \boldsymbol{X}_2^{k+1}]\times_3 [ \boldsymbol{X}_3^{k+1}],
\end{equation}
$\mathcal{I}\in \mathbb{R}^{I_{1}\times I_{2}\times I_{3}}$ denotes the tensor with all entries equal to $1$, and $\mathcal{W}_\mathscr{R}^{-1}$ denotes the component-wise inverse of $\mathcal{W}_\mathscr{R}$, i.e., the $(i_1,i_2,i_3)$-th entry of $\mathcal{W}_\mathscr{R}^{-1}$ is equal to $1/(\mathcal{W}_\mathscr{R})_{i_1i_2i_3}$.

All in all, the proposed P-BCD algorithm for model~\eqref{model:HSI} is summarized in Algorithm~\ref{alg1}.
\begin{algorithm}[H]
	\renewcommand{\algorithmicrequire}{\textbf{Input:}}
	\renewcommand{\algorithmicensure}{\textbf{Output:}}
	\caption{Proximal BCD (P-BCD) algorithm for model~\eqref{model:HSI}}
	\label{alg1}
	\begin{algorithmic}[1]
		\STATE Initialize $(\mathcal{S}^0,[ \boldsymbol{X}_1^0],[ \boldsymbol{X}_2^0],[ \boldsymbol{X}_3^0],[\mathcal{G}^0],\mathcal{L}^0)$ with \([ \boldsymbol{X}_i^0]\in [\mathbb{S}_{m_i,n_i}]\);
  \STATE Set the tensor extraction operator $\mathscr{R}$;
		\STATE  Set parameters $\alpha_{\mathcal{S}}, \alpha_{ \boldsymbol{X}}, \alpha_{\mathcal{G}}>0$;
  \STATE Set $k=0$.
		\REPEAT
		\STATE Compute $\mathcal{S}^{k+1}$ by \eqref{eq:Supdate};
  \STATE Compute $[ \boldsymbol{X}_i^{k+1}]$ by \eqref{eq:Xupdate},  $i=1,2,3$;
  \STATE Compute $[\mathcal{G}^{k+1}]$ by \eqref{eq:Gupdate};
\STATE Compute $\mathcal{L}^{k+1}$ by \eqref{eq:Lupdate};
		\STATE $k\leftarrow k+1$.
		\UNTIL the stopping criterion is met.
		\ENSURE $(\mathcal{S}^k,[ \boldsymbol{X}_1^k],[ \boldsymbol{X}_2^k],[ \boldsymbol{X}_3^k],[\mathcal{G}^k],\mathcal{L}^k)$. \end{algorithmic}
\end{algorithm}

% % -------------------- 4. Convergence Analysis Section -------------------- % %

\section{Convergence Analysis of The P-BCD Algorithm}\label{sec:convergence}
The proposed P-BCD algorithm aims to solve a particular optimization problem of the form as in~\eqref{model:HSI}.  An optimal solution of each subproblem  in our algorithm is obtained as shown in the previous subsections.
The P-BCD algorithm can be viewed as a special variant of the proximal alternating linearization minimization (PALM)~\cite{Bolte2014} extended for multiple blocks or the block coordinate update with prox-linear approximation~\cite{xu2017globally}, called the block prox-linear method. In the following, we present the convergence results of the P-BCD algorithm.

Let \(\mathcal{Z}:=(\mathcal{S},[ \boldsymbol{X}_1],[ \boldsymbol{X}_2],[ \boldsymbol{X}_3],[\mathcal{G}],\mathcal{L})\). First, we define the first-order optimality condition of the orthogonal constrained optimization problem as in \eqref{model:HSI}. The point  $\bar{\mathcal{Z}}:=(\bar{\mathcal{S}},[\bar{ \boldsymbol{X}}_1],[\bar{ \boldsymbol{X}}_2],[\bar{ \boldsymbol{X}}_3],[\bar{\mathcal{G}}],\bar{\mathcal{L}})$ is a first-order stationary point of problem \eqref{model:HSI} if $\boldsymbol{0}\in \partial \Phi(\bar{\mathcal{Z}})$, that is,
\begin{align}
&\boldsymbol{0}\in
	\bar{\mathcal{L}}+
	\bar{\mathcal{S}}-\mathcal{D} + \gamma
	\partial\varphi(\bar{\mathcal{S}}),\nonumber\\
&\boldsymbol{0}=\operatorname{grad}_{[ \boldsymbol{X}_i]} H([\bar{ \boldsymbol{X}}_1],[\bar{ \boldsymbol{X}}_2],[\bar{ \boldsymbol{X}}_3],[\bar{\mathcal{G}}],\bar{\mathcal{L}}),\quad [\bar{\boldsymbol{X}}_i]^{\top}[\bar{ \boldsymbol{X}}_i]=[\boldsymbol{I}_{n_i}], i=1,2,3,\label{eq:gradXH}\\
&\boldsymbol{0}\in \delta \nabla_{[\mathcal{G}]}H([\bar{\boldsymbol{X}}_1],[\bar{\boldsymbol{X}}_2],[\bar{\boldsymbol{X}}_3],[\bar{\mathcal{G}}],\bar{\mathcal{L}})+\partial \psi ([\bar{\mathcal{G}}]),\nonumber\\
&\boldsymbol{0}= \bar{\mathcal{L}}+\bar{\mathcal{S}}-\mathcal{D}+ \delta\nabla_{\mathcal{L}}H([\bar{\boldsymbol{X}}_1],[\bar{\boldsymbol{X}}_2],[\bar{\boldsymbol{X}}_3],[\bar{\mathcal{G}}],\bar{\mathcal{L}}),\nonumber
\end{align}
where  $\operatorname{grad}_{[\boldsymbol{X}_i]} H([\bar{\boldsymbol{X}}_1],[\bar{\boldsymbol{X}}_2],[\bar{\boldsymbol{X}}_3],[\bar{\mathcal{G}}],\bar{\mathcal{L}})$ denotes the Riemannian gradient of $H$ with respect to $[\boldsymbol{X}_i]$ evaluated at $([\bar{\boldsymbol{X}}_1],[\bar{\boldsymbol{X}}_2],[\bar{\boldsymbol{X}}_3],[\bar{\mathcal{G}}],\bar{\mathcal{L}})$, $ i=1,2,3$, and $\partial\varphi$ and $\partial \psi$  denote the subdifferentials of $\varphi$ and $\psi$, respectively. Then we compute the gradients of $H$ explicitly and replace the optimality condition for orthogonal constraints as in \eqref{eq:gradXH} using an equivalent condition introduced in \cite{Gao2018A}. Hence, we call $\bar{\mathcal{Z}}$ is a first-order stationary point of problem \eqref{model:HSI} if
\begin{subequations}
\begin{align}
&\boldsymbol{0}\in \bar{\mathcal{L}}+
	\bar{\mathcal{S}}-\mathcal{D} + \gamma
	\partial\varphi(\bar{\mathcal{S}}),\label{Eq:optimala}\\
&\boldsymbol{0}=([\boldsymbol{I}_{m_i}]-[\bar{\boldsymbol{X}}_i][\bar{\boldsymbol{X}}_i]^{\top})[\bar{\boldsymbol{H}}_i], \label{Eq:optimalb}\\
&\boldsymbol{0}=[\bar{ \boldsymbol{H}}_i]^{\top}[\bar{ \boldsymbol{X}}_i]-[\bar{ \boldsymbol{X}}_i]^{\top}[\bar{ \boldsymbol{H}}_i], \label{Eq:optimalc}\\
&[\bar{ \boldsymbol{X}}_i]^{\top}[\bar{ \boldsymbol{X}}_i]=[\mathbf{I}_{n_i}],\label{Eq:optimald}\\
&\boldsymbol{0}\in \delta([\bar{\mathcal{G}}]-[\bar{\mathcal{O}}])+\partial \psi ([\bar{\mathcal{G}}]),\label{Eq:optimale}\\
&\boldsymbol{0}= \bar{\mathcal{L}}+\bar{\mathcal{S}}-\mathcal{D}+\delta \mathcal{W}_{\mathscr{R}}\odot\bar{\mathcal{L}}-\delta\mathscr{R}^{\top}([\bar{\mathcal{Y}}]),\label{Eq:optimalf}
\end{align}
\end{subequations}
where $ i=1,2,3$,  and
\begin{align*}
    [\bar{\boldsymbol{P}}_{i}]=&(\mathscr{R}(\bar{\mathcal{L}}))_{[(i)]}\\
    [\bar{\boldsymbol{Q}}_{1}]=& ([\bar{\mathcal{G}}]\times_2[ \bar{\boldsymbol{X}}_2]\times_3[ \bar{\boldsymbol{X}}_3])_{[(1)]}\\
    [\bar{\boldsymbol{Q}}_{2}]=& ([\bar{\mathcal{G}}]\times_1[ \bar{\boldsymbol{X}}_1]\times_3[ \bar{\boldsymbol{X}}_3])_{[(2)]}\\
    [\bar{\boldsymbol{Q}}_{3}]=& ([\bar{\mathcal{G}}]\times_1[ \bar{\boldsymbol{X}}_1]\times_2[ \bar{\boldsymbol{X}}_2])_{[(3)]}\\
    [\bar{\boldsymbol{H}}_i]=&([\bar{\boldsymbol{X}}_i][\bar{\boldsymbol{Q}}_{i}]-[\bar{\boldsymbol{P}}_{i}]) [\bar{\boldsymbol{Q}}_{i}]^{\top}\\
    [\bar{\mathcal{O}}]=&\mathscr{R}(\bar{\mathcal{L}})\times_1 [\bar{\boldsymbol{X}}_1]^{\top}\times_2 [\bar{\boldsymbol{X}}_2]^{\top}\times_3 [\bar{\boldsymbol{X}}_3]^{\top}\\
    [\bar{\mathcal{Y}}]=&[\bar{\mathcal{G}}]\times_1 [\bar{\boldsymbol{X}}_1]\times_2 [\bar{\boldsymbol{X}}_2]\times_3 [\bar{\boldsymbol{X}}_3].
\end{align*}

Second, we prove the non-increasing monotonicity of the objective sequence\\ $\{\Phi(\mathcal{Z}^k)\}$ and the boundedness of the sequence $\{\mathcal{Z}^k\}$ generated by Algorithm~\ref{alg1}.

\begin{theorem}\label{Thm:IterateSeq}
Let $\{\mathcal{Z}^k\}$ be the sequence generated by Algorithm~\ref{alg1}. Then the following statements hold:
\begin{enumerate}
    \item[(i)]  The sequence $\{\Phi(\mathcal{Z}^k)\}$ of function values at the iteration points decreases monotonically, and
\begin{equation}\label{Eq:Decrease}
    \begin{split}
&\Phi(\mathcal{Z}^k)-\Phi(\mathcal{Z}^{k+1})\\
\ge &\frac{\alpha}{2}\|\mathcal{S}^{k+1}-\mathcal{S}^{k}\|^{2}_{F}+\frac{\alpha}{2}\sum_{i=1}^{3}\|[\boldsymbol{X}_i^{k+1}]-[\boldsymbol{X}_i^{k}]\|^{2}_{F}+\frac{\alpha}{2}\|[\mathcal{G}^{k+1}]-[\mathcal{G}^{k}]\|^{2}_{F}.
    \end{split}
\end{equation}
\item[(ii)] The sequence $\{\mathcal{Z}^{k}\}$ is bounded.

\item[(iii)] $\displaystyle\lim_{k\to\infty} \|\mathcal{S}^{k+1}-\mathcal{S}^{k}\|_{F}=0$, $\displaystyle\lim_{k\to\infty}\|[\boldsymbol{X}_i^{k+1}]-[\boldsymbol{X}_i^{k}]\|_{F}=0$,   for $i=1,2,3$, and $\displaystyle\lim_{k\to\infty}\|[\mathcal{G}^{k+1}]-[\mathcal{G}^{k}]\|_{F}=0$.
\end{enumerate}
\end{theorem}

\begin{proof} (i) According to the update of $\mathcal{S}$, we have
\begin{align*}
    &\Phi(\mathcal{Z}^k)-\Phi(\mathcal{S}^{k+1},[\boldsymbol{X}_1^k],[\boldsymbol{X}_2^k],[\boldsymbol{X}_3^k],[\mathcal{G}^k],\mathcal{L}^k)\\
    \ge& \frac{\alpha_{\mathcal{S}}}{2}\|\mathcal{S}^{k+1}-\mathcal{S}^{k}\|^{2}_{F},
\end{align*}
where $\alpha>0$.
Next, it follows from the update of $[\boldsymbol{X}_i]$ and Lemma~\ref{Thm:Stiefel} that
\begin{align*}
&\Phi(\mathcal{S}^{k+1},[\boldsymbol{X}_1^k],[\boldsymbol{X}_2^k],[\boldsymbol{X}_3^k],[\mathcal{G}^k],\mathcal{L}^k)-\Phi(\mathcal{S}^{k+1},[\boldsymbol{X}_1^{k+1}],[\boldsymbol{X}_2^{k+1}],[\boldsymbol{X}_3^{k+1}],[\mathcal{G}^k],\mathcal{L}^k)\\
\geq & \frac{\alpha_{\boldsymbol{X}}}{2}\sum_{i=1}^{3}\|[\boldsymbol{X}_i^{k+1}]-[\boldsymbol{X}_i^{k}]\|^{2}_{F}.
\end{align*}
Then by the updates of $[\mathcal{G}]$ and $\mathcal{L}$, we have
\begin{align*}
&\Phi(\mathcal{S}^{k+1},[\boldsymbol{X}_1^{k+1}],[\boldsymbol{X}_2^{k+1}],[\boldsymbol{X}_3^{k+1}],[\mathcal{G}^{k}],\mathcal{L}^k)\\
&~~~~~~~-\Phi(\mathcal{S}^{k+1},[\boldsymbol{X}_1^{k+1}],[\boldsymbol{X}_2^{k+1}],[\boldsymbol{X}_3^{k+1}],[\mathcal{G}^{k+1}],\mathcal{L}^k)\ge\frac{\alpha_{\mathcal{G}}}{2}\|[\mathcal{G}^{k+1}]-[\mathcal{G}^{k}]\|^{2}_{F}
\end{align*}
and
\begin{displaymath}
    \Phi(\mathcal{S}^{k+1},[\boldsymbol{X}_1^{k+1}],[\boldsymbol{X}_2^{k+1}],[\boldsymbol{X}_3^{k+1}],[\mathcal{G}^{k+1}],\mathcal{L}^k)-\Phi(\mathcal{Z}^{k+1})\ge 0.
\end{displaymath}
Combining the inequalities above, we  obtain \eqref{Eq:Decrease} with $\alpha =\min\{\alpha_{\mathcal{S}},\alpha_{\mathcal{G}},\alpha_{\boldsymbol{X}}\}>0$.

(ii) Since $[\boldsymbol{X}_i^k]^{\top}[\boldsymbol{X}_i^k]=[\boldsymbol{I}_{n_i}]$ for each $i=1,2,3$, we have the sequence $\{[\boldsymbol{X}_i^k]\}$ is bounded.
By (i), we have $\Phi(\mathcal{Z}^k)\le \Phi(\mathcal{Z}^0)$. Also, we observe that $\Phi(\mathcal{Z}^k)\ge \gamma \varphi(\mathcal{S}^k)+\psi([\mathcal{G}^k])\ge 0$. Since
\begin{displaymath}
    \lim_{\|\mathcal{S}\|_F\to \infty}\varphi(\mathcal{S})=\infty \quad\mbox{ and }\quad \lim_{ \|[\mathcal{G}]\|_F\to \infty} \psi([\mathcal{G}])=\infty,
\end{displaymath}
we must have the sequences $\{\mathcal{S}^k\}$ and $\{[\mathcal{G}^k]\}$ are bounded. As shown in \eqref{eq:Lupdate} that $\mathcal{L}^k$ is uniquely determined by  $\mathcal{S}^k,[\boldsymbol{X}_1^k],[\boldsymbol{X}_2^k],[\boldsymbol{X}_3^k]$ and $[\mathcal{G}^k]$, the sequence $\{\mathcal{L}^k\}$ is also bounded.

(iii) Let $K$ be an arbitrary integer. Summing \eqref{Eq:Decrease} from $k=1$ to $K-1$, we have
 \begin{align*}
 &\sum^{K-1}_{k=0}\|\mathcal{S}^{k+1}-\mathcal{S}^{k}\|^{2}_{F}+ \sum^{K-1}_{k=0}\sum_{i=1}^{3}\|[\boldsymbol{X}_i^{k+1}]-[\boldsymbol{X}_i^{k}]\|^{2}_{F}+ \sum^{K-1}_{k=0}\|[\mathcal{G}^{k+1}]-[\mathcal{G}^{k}]\|^{2}_{F}\\
 \le& \frac{2}{\alpha}\left(\Phi(\mathcal{Z}^0)-\Phi(\mathcal{Z}^K)\right)\\
 \le&\frac{2}{\alpha}\Phi(\mathcal{Z}^0).
\end{align*}
 Taking the limits of both sides of the inequality as $K\rightarrow \infty$, we have $\sum^{\infty}_{k=0}\|\mathcal{S}^{k+1}-\mathcal{S}^{k}\|^{2}_{F}<\infty$, $\sum_{k=0}^{\infty}\|[\boldsymbol{X}_i^{k+1}]-[\boldsymbol{X}_i^{k}]\|^{2}_{F}<\infty$ and $\sum_{k=0}^{\infty}\|[\mathcal{G}^{k+1}]-[\mathcal{G}^{k}]\|^{2}_{F}<\infty$. Then assertion (iii) immediately holds.
\end{proof}

In addition to the assertions presented in Theorem~\ref{Thm:IterateSeq}, more assertions can be derived in the following corollary.

\begin{corollary}~\label{Cor:L} Let $\{\mathcal{Z}^k\}$ be the sequence generated by Algorithm~\ref{alg1}. Then $\displaystyle\lim_{k\to\infty}\|[\mathcal{Y}^{k+1}]-[\mathcal{Y}^{k}]\|_{F}=0$ and
 $\displaystyle\lim_{k\to\infty} \|\mathcal{L}^{k+1}-\mathcal{L}^{k}\|_{F}=0$.
\end{corollary}
\begin{proof} Since $\|[\mathcal{G}^{k+1}]\times
_{i}[\boldsymbol{X}_i]-[\mathcal{G}^{k}]\times
_{i}[\boldsymbol{X}_i]\|_F=\|[\mathcal{G}^{k+1}]-[\mathcal{G}^{k}]\|_F$ for any $[\boldsymbol{X}_i]\in [\mathbb{S}_{m_i,n_i}]$ and $\|[\mathcal{G}]\times
_{i}[\boldsymbol{X}_i^{k+1}]-[\mathcal{G}]\times
_{i}[\boldsymbol{X}_i^{k}]\|_F\leq \|[\mathcal{G}]\|_F\|[\boldsymbol{X}_i^{k+1}]-[\boldsymbol{X}_i^{k}]\|_F$, we have
\begin{align*}
    \|[\mathcal{Y}^{k+1}]-[\mathcal{Y}^{k}]\|_F
    \leq  &\|[\mathcal{G}^{k+1}]-[\mathcal{G}^{k}]\|_F+ c_1\|[\boldsymbol{X}_1^{k+1}]-[\boldsymbol{X}_1^{k}]\|_{F}\\
    &+c_1\|[\boldsymbol{X}_2^{k+1}]-[\boldsymbol{X}_2^{k}]\|_{F}+c_1\|[\boldsymbol{X}_3^{k+1}]-[\boldsymbol{X}_3^{k}]\|_{F},
\end{align*}
where $c_1=\max_k \|[\mathcal{G}^k]\|_F<\infty$ according to assertion (ii) in Theorem~\ref{Thm:IterateSeq}. Then it immediately follows from assertion (iii) in Theorem~\ref{Thm:IterateSeq} that  $\lim_{k\to\infty}\|[\mathcal{Y}^{k+1}]-[\mathcal{Y}^{k}]\|_{F}=0$.

Also, we have
\begin{align*}
    \|\mathcal{L}^{k+1}-\mathcal{L}^{k}\|_F
    \leq  &\delta \sqrt{c_2}\|[\mathcal{Y}^{k+1}]-[\mathcal{Y}^{k}]\|_F + \|\mathcal{S}^{k+1}-\mathcal{S}^{k}\|_F,
\end{align*}
where $c_2=\max (\mathcal{W}_\mathscr{R})_{i_1i_2i_3}$. Then by assertion (iii) in Theorem~\ref{Thm:IterateSeq}, we can have $\lim_{k\to\infty}\|\mathcal{L}^{k+1}-\mathcal{L}^{k}\|_{F}=0$.
 \end{proof}

Third, we apply the results in \cite{Gao2018A} to the updates of $[\boldsymbol{X}_i]$ in our proposed algorithm. Those results are useful for proving three equalities of substationarity, symmetry, and feasibility for $[\boldsymbol{X}_i]$. According to Lemma 3.3 in \cite{Gao2018A}, we can have the following lemma and then we prove the decrease of the function value $H$ after each update of $[\boldsymbol{X}_i]$.
 \begin{lemma}[\cite{Gao2018A}]\label{Lem:ProjSmn}
     Let $h:\mathbb{R}^{m\times n}\to \mathbb{R}$ be defined by $h(\boldsymbol{X})=\frac{1}{2}\|\boldsymbol{X}-\boldsymbol{P}\boldsymbol{Q}^\top\|_F^2$, where $\boldsymbol{Q}\in \mathbb{R}^{n\times m}$, $\boldsymbol{P}\in \mathbb{R}^{m\times m}$ and $m\geq n$.  If $\boldsymbol{X}\in \mathbb{S}_{m,n}$ and \begin{equation}\label{eq:XbarUpdate}\bar{X}=\operatorname{Proj}_{\mathbb{S}_{m,n}}(\boldsymbol{X}-\tau (\boldsymbol{X}-\boldsymbol{P}\boldsymbol{Q}^{\top})),\end{equation}
     where $\tau\in(0,1)$, then we have $\bar{\boldsymbol{X}}\in \mathbb{S}_{m,n}$ and
     \begin{equation}
         h(\boldsymbol{X})-h(\bar{\boldsymbol{X}})\ge \frac{\tau^{-1}-1}{2(\tau^{-1}+1+\theta)^2}\|(\boldsymbol{I}_{m} -\boldsymbol{X}{\boldsymbol{X}}^{\top})\nabla h(\boldsymbol{X})\|_F^2,
     \end{equation}
     where $\theta=\|\boldsymbol{P}\boldsymbol{Q}^\top\|_2$.
 \end{lemma}

\begin{proposition}\label{Thm:XX}
Let $H$ be defined as in \eqref{Eq:DefH}. Let $\{\mathcal{Z}^k\}$ be the sequence generated by Algorithm~\ref{alg1} and $[\boldsymbol{X}_i^{k+1}]$ be computed by \eqref{eq:Xupdate}. Then $[\boldsymbol{X}_i^{k+1}]\in [\mathbb{S}_{m_i,n_i}]$ and the  following inequality holds:
\begin{equation}\label{eq:H_ineq1}
\begin{split}
    &H([\boldsymbol{X}_1^k],[\boldsymbol{X}_2^k],[\boldsymbol{X}_3^k],[\mathcal{G}^k],\mathcal{L}^k)-H([\boldsymbol{X}_1^{k+1}],[\boldsymbol{X}_2^{k+1}],[\boldsymbol{X}_3^{k+1}],[\mathcal{G}^k],\mathcal{L}^k)\\
\geq& c_3\sum_{i=1}^{3}\left\| ([\boldsymbol{I}_{m_{i}}]-[\boldsymbol{X}_i^{k}]{[\boldsymbol{X}_i^{k}]}^{\top})[\boldsymbol{H}_i^k]\right\|_F^2
\end{split}
\end{equation}
and
\begin{equation}\label{eq:H_ineq2}
    \left\|[\boldsymbol{H}_i^k]^{\top}[\boldsymbol{X}_i^{k}]-{[\boldsymbol{X}_i^{k}]}^{\top}[H_i^k]\right\|_F
    \leq c_4\|[\boldsymbol{X}_{i+1}^{k}]-[\boldsymbol{X}_{i}^{k}]\|_F,
\end{equation}
where $c_3=\frac{{\tilde{\alpha}_{\boldsymbol{X}}}^{-1}-1}{2({\tilde{\alpha}_{\boldsymbol{X}}}^{-1}+1+\theta_{\max})^2}>0$,  
$c_4=2(\alpha_{\boldsymbol{X}}\sqrt{n_{\max}}+\theta_{\max})>0$, $\theta_{\max}=\max_{ijk}$ $ \|[\boldsymbol{P}_{i}^k]^{(j)}([\boldsymbol{Q}_{i}^k]^{(j)})^{\top}\|_2$, $n_{\max}=\max \{n_1,n_2,n_3\}$, and
\begin{align*}
[\boldsymbol{H}_1^k]&=\nabla_{[\boldsymbol{X}_1]} H([\boldsymbol{X}_1^{k}], [\boldsymbol{X}_2^{k}],[\boldsymbol{X}_3^k],[\mathcal{G}^k],\mathcal{L}^k)\\
[\boldsymbol{H}_2^k]&=\nabla_{[\boldsymbol{X}_2]} H([\boldsymbol{X}_1^{k+1}], [\boldsymbol{X}_2^{k}],[\boldsymbol{X}_3^k],[\mathcal{G}^k],\mathcal{L}^k)\\
[\boldsymbol{H}_3^k]&=\nabla_{[\boldsymbol{X}_3]} H([\boldsymbol{X}_1^{k+1}], [\boldsymbol{X}_2^{k+1}],[\boldsymbol{X}_3^k],[\mathcal{G}^k],\mathcal{L}^k)
\end{align*}
which give
$$[\boldsymbol{H}_i^k]=([\boldsymbol{X}_i^k][\boldsymbol{Q}_{i}^k]-[\boldsymbol{P}_{i}^k]) [\boldsymbol{Q}_{i}^k]^{\top},\quad i=1,2,3.$$

\end{proposition}

\begin{proof}
It follows from Lemma~\ref{Lem:ProjSmn} that  $[\boldsymbol{X}_i^{k+1}]\in \mathbb{S}_{m_i,n_i}$.

To show the first inequality holds, the update of $[\boldsymbol{X}_i^{k+1}]$ in \eqref{eq:Xupdate} can be viewed as the iteration in \eqref{eq:XbarUpdate} with $h([\boldsymbol{X}])=\frac{1}{2}\|[\boldsymbol{X}]-[\boldsymbol{P}_{i}^k][\boldsymbol{Q}_{i}^k]^{\top}\|_F^2$, $[\boldsymbol{X}]=[\boldsymbol{X}_i^{k}]$, $[\bar{\boldsymbol{X}}]=[\boldsymbol{X}_i^{k+1}]$ and $\tau=\tilde{\alpha}_{\boldsymbol{X}}=\frac{\delta}{\delta+\alpha_{\boldsymbol{X}}}$. Then for $i=1$ we have
\begin{align*}
&H([\boldsymbol{X}_1^{k}],[\boldsymbol{X}_{2}^{k}],[\boldsymbol{X}_3^k],[\mathcal{G}^k],\mathcal{L}^k)\\
&-H([\boldsymbol{X}_1^{k+1}],[\boldsymbol{X}_{2}^{k}],[\boldsymbol{X}_3^k],[\mathcal{G}^k],\mathcal{L}^k)\\
=&\frac{1}{2}\|[\boldsymbol{X}_i^{k+1}][\boldsymbol{Q}_{i}^k]-[\boldsymbol{P}_{i}^k]\|_F^2-\frac{1}{2}\|[\boldsymbol{X}_i^{k}][\boldsymbol{Q}_{i}^k]-[\boldsymbol{P}_{i}^k]\|_F^2\\
=&\frac{1}{2}\|[\boldsymbol{X}_i^{k+1}]-[\boldsymbol{P}_{i}^k][\boldsymbol{Q}_{i}^k]^{\top}\|_F^2-\frac{1}{2}\|[\boldsymbol{X}_i^{k}]-[\boldsymbol{P}_{i}^k][\boldsymbol{Q}_{i}^k]^{\top}\|_F^2\\
\geq &c_{ij}^{k} \left\|\left([\mathbf{I}_{m_i}] -[\boldsymbol{X}_i^{k}]{[\boldsymbol{X}_i^{k}]}^{\top}\right)\left([\boldsymbol{X}_i^{k}]-[\boldsymbol{P}_{i}^k][\boldsymbol{Q}_{i}^k]^{\top}\right)\right\|_F^2,
\end{align*}
where $c_{ij}^{k}=\frac{{\tilde{\alpha}_{X}}^{-1}-1}{2({\tilde{\alpha}_{\boldsymbol{X}}}^{-1}+1+\|[\boldsymbol{P}_{i}^k]^{(j)}([\boldsymbol{Q}_{i}^k]^{(j)})^{\top}\|_2)^2}\geq c_3>0$ and $\theta_{\max}$ is bounded, since the sequence $\{\mathcal{Z}^k\}$ is bounded. Using $[\boldsymbol{X}_i^{k}]\in [\mathbb{S}_{m_i,n_i}]$ and $[\boldsymbol{X}_i^{k+1}]\in [\mathbb{S}_{m_i,n_i}]$, we can obtain the fourth line above and rewrite part of the last line as follows
\begin{align*}
   &\left([\mathbf{I}_{m_i}] -[\boldsymbol{X}_i^{k}]{[\boldsymbol{X}_i^{k}]}^{\top}\right)\left([\boldsymbol{X}_i^{k}]-[\boldsymbol{P}_{i}^k][\boldsymbol{Q}_{i}^k]^{\top}\right)\\
   =&\left([\mathbf{I}_{m_i}] -[\boldsymbol{X}_i^{k}]{[\boldsymbol{X}_i^{k}]}^{\top}\right)\left([\boldsymbol{X}_i^{k}][\boldsymbol{Q}_{i}^k]-[\boldsymbol{P}_{i}^k]\right)[\boldsymbol{Q}_{i}^k]^{\top}\\
   =&\left([\boldsymbol{I}_{m_i}] -[\boldsymbol{X}_i^{k}]{[\boldsymbol{X}_i^{k}]}^{\top}\right)[\boldsymbol{H}_i^k].
\end{align*}
 That is, we have
 \begin{equation*}
\begin{split}
    &H([\boldsymbol{X}_1^{k}],[\boldsymbol{X}_{2}^{k}],[\boldsymbol{X}_3^k],[\mathcal{G}^k],\mathcal{L}^k)-H([\boldsymbol{X}_1^{k+1}], [\boldsymbol{X}_{2}^{k}],[\boldsymbol{X}_3^k],[\mathcal{G}^k],\mathcal{L}^k)\\
\geq& c_3\left\|\left([\mathbf{I}_{m_1}] -[\boldsymbol{X}_1^{k}]{[\boldsymbol{X}_1^{k}]}^{\top}\right)[\boldsymbol{H}_1^k]\right\|_F^2.
\end{split}
\end{equation*}
Similarly, we have
 \begin{equation*}
\begin{split}
&H([\boldsymbol{X}_1^{k+1}],[\boldsymbol{X}_{2}^{k}],[\boldsymbol{X}_3^k],[\mathcal{G}^k],\mathcal{L}^k)
-H([\boldsymbol{X}_1^{k+1}], [\boldsymbol{X}_{2}^{k+1}],[\boldsymbol{X}_3^k],[\mathcal{G}^k],\mathcal{L}^k)\\
\geq& c_3\left\|\left([\boldsymbol{I}_{m_2}] -[\boldsymbol{X}_2^{k}]{[\boldsymbol{X}_2^{k}]}^{\top}\right)[\boldsymbol{H}_2^k]\right\|_F^2
\end{split}
\end{equation*}
and
 \begin{equation*}
\begin{split}
&H([\boldsymbol{X}_1^{k+1}],[\boldsymbol{X}_{2}^{k+1}],[\boldsymbol{X}_3^k],[\mathcal{G}^k],\mathcal{L}^k)
-H([\boldsymbol{X}_1^{k+1}], [\boldsymbol{X}_{2}^{k+1}],[\boldsymbol{X}_3^{k+1}],[\mathcal{G}^k],\mathcal{L}^k)\\
\geq & c_3\left\|\left([\boldsymbol{I}_{m_3}] -[\boldsymbol{X}_3^{k}]{[\boldsymbol{X}_3^{k}]}^{\top}\right)[\boldsymbol{H}_3^k]\right\|_F^2.
\end{split}
\end{equation*}
Summing the inequalities above, \eqref{eq:H_ineq1} immediately holds.

Next, we show the second inequality holds. According to the update of $[\boldsymbol{X}_i^{k+1}]$ in \eqref{eq:Xupdate}, if we let
\begin{align*}
	[\boldsymbol{U}^k][\boldsymbol{\Sigma}^k] [\boldsymbol{V}^k]^{\top} &= [\boldsymbol{X}_i^{k}]- \tilde{\alpha}_{\boldsymbol{X}}\left([\boldsymbol{X}_i^{k}]-[\boldsymbol{P}_{i}^k][\boldsymbol{Q}_{i}^k]^{\top}\right)\\
	&=\frac{1}{\alpha_{\boldsymbol{X}}+\delta}\left(\alpha_{\boldsymbol{X}}[\boldsymbol{X}_i^k]+\delta[\boldsymbol{P}_{i}^k][\boldsymbol{Q}_{i}^k]^{\top}\right),
\end{align*}
 then the update $[\boldsymbol{X}_i^{k+1}]$ $=[\boldsymbol{U}^k][\boldsymbol{V}^k]^{\top}$. Hence, we have
\begin{align}
[\boldsymbol{X}_i^{k+1}]^{\top}\left(\alpha_{\boldsymbol{X}}[\boldsymbol{X}_i^{k}]+\delta[\boldsymbol{P}_{i}^k][\boldsymbol{Q}_{i}^k]^{\top}\right)&=\left(\alpha_{\boldsymbol{X}}[\boldsymbol{X}_i^{k}]+\delta[\boldsymbol{P}_{i}^k][\boldsymbol{Q}_{i}^k]^{\top}\right)^{\top}[\boldsymbol{X}_i^{k+1}]\label{eq:xk+1pq}\\
&=(\alpha_{\boldsymbol{X}}+\delta)[\boldsymbol{V}]^k[\boldsymbol{\Sigma}^k][\boldsymbol{V}^k]^{\top}.\nonumber
\end{align}
Then using $[\boldsymbol{X}_i^{k}]\in [\mathbb{S}_{m_i,n_i}]$ and $[\boldsymbol{X}_i^{k+1}]\in [\mathbb{S}_{m_i,n_i}]$,  we can rewrite
\begin{align*}
    &[\boldsymbol{H}_i^k]^{\top}[\boldsymbol{X}_i^{k}]-{[\boldsymbol{X}_i^{k}]}^{\top}[\boldsymbol{H}_i^k]\\
    =&\left(\alpha_{\boldsymbol{X}}[\boldsymbol{X}_i^{k}]+\delta[\boldsymbol{P}_{i}^k][\boldsymbol{Q}_{i}^k]^{\top}\right)^{\top}[\boldsymbol{X}_i^{k}]-[\boldsymbol{X}_i^{k}]^{\top}\left(\alpha_{\boldsymbol{X}} [\boldsymbol{X}_i^{k}]+\delta[\boldsymbol{P}_{i}^k][\boldsymbol{Q}_{i}^k]^{\top}\right)\\
    =&\left([\boldsymbol{X}_i^{k+1}]-[\boldsymbol{X}_i^{k}]\right)^{\top}\left(\alpha_{\boldsymbol{X}}[\boldsymbol{X}_i^{k}]+\delta[\boldsymbol{P}_{i}^k][\boldsymbol{Q}_{i}^k]^{\top}\right)\\
    &-\left(\alpha_{\boldsymbol{X}}[\boldsymbol{X}_i^{k}]+\delta[\boldsymbol{P}_{i}^k][\boldsymbol{Q}_{i}^k]^{\top}\right)^{\top}\left([\boldsymbol{X}_i^{k+1}]-[\boldsymbol{X}_i^{k}]\right),
\end{align*}
where the last equation is obtained by  \eqref{eq:xk+1pq}. Taking the Frobenius norm of both sides, we obtain \eqref{eq:H_ineq2}, where $2\|\alpha_{\boldsymbol{X}}[\boldsymbol{X}_i^{k}]^{(j)}+\delta[\boldsymbol{P}_{i}^k]^{(j)}([\boldsymbol{Q}_{i}^k]^{(j)})^{\top}\|_F\leq c_4$.
\end{proof}

Lastly, we show that every convergent subsequence converges to a first-order stationary point of problem \eqref{model:HSI}.

\begin{theorem}
Let $\{\mathcal{Z}^k\}$ be the sequence generated by Algorithm~\ref{alg1}. Then every accumulation point of $\{\mathcal{Z}^k\}$ is a first-order stationary point of problem \eqref{model:HSI}.
\end{theorem}
\begin{proof} Suppose that $\{\mathcal{Z}^{\tilde{k}}\}_{\tilde{k}\in\mathcal{K}}$ is a convergent subsequence of $\{\mathcal{Z}^k\}$ and converges to $\bar{\mathcal{Z}}$ as $\tilde{k}\in \mathcal{K}$ approaches $\infty$. By the updates of $\mathcal{S}^{\tilde{k}}$, $[\mathcal{G}^{\tilde{k}}]$ and $\mathcal{L}^{\tilde{k}}$, we have for any $\tilde{k}\in \mathcal{K}$,
\begin{align*}
	\alpha_{\mathcal{S}}(\mathcal{S}^{\tilde{k}-1}-\mathcal{S}^{\tilde{k}})+&(\mathcal{L}^{\tilde{k}}-\mathcal{L}^{\tilde{k}-1})\\
		&\in \mathcal{L}^{\tilde{k}}+\mathcal{S}^{\tilde{k}}-\mathcal{D} +\gamma \partial \varphi(\mathcal{S}^{\tilde{k}}),
\end{align*}
and
\begin{align*}
&\alpha_{\mathcal{G}}([\mathcal{G}^{\tilde{k}-1}]-[\mathcal{G}^{\tilde{k}}])+\delta([\mathcal{O}^{\tilde{k}}]-[\tilde{\mathcal{O}}^{\tilde{k}}])\in \delta([\mathcal{G}^{\tilde{k}}]-[\tilde{\mathcal{O}}^{\tilde{k}}])+\partial \psi ([\mathcal{G}^{\tilde{k}}]),\\
&{0}= \mathcal{L}^{\tilde{k}}+\mathcal{S}^{\tilde{k}}-\mathcal{D}+\delta \mathcal{W}_{\mathscr{R}}\odot\mathcal{L}^{\tilde{k}}-\delta\mathscr{R}^{\top}([\mathcal{Y}^{\tilde{k}}]),
\end{align*}
where $[\tilde{\mathcal{O}}^{\tilde{k}}]=\mathscr{R}(\mathcal{L}^{\tilde{k}})\times_1 [\boldsymbol{X}_1^{\tilde{k}}]^{\top}\times_2 [\boldsymbol{X}_2^{\tilde{k}}]^{\top}\times_3 [\boldsymbol{X}_3^{\tilde{k}}]^{\top}$. Since we have
\begin{align*}
    \|[\mathcal{O}^{\tilde{k}}]-[\tilde{\mathcal{O}}^{\tilde{k}}]\|_F&\leq \sqrt{c_2}\|[\boldsymbol{X}_1^{\tilde{k}}]\|_F\|[\boldsymbol{X}_2^{\tilde{k}}]\|_F\|[\boldsymbol{X}_3^{\tilde{k}}]\|_F\|\mathcal{L}^{\tilde{k}}-\mathcal{L}^{\tilde{k}-1}\|_F\\
    &\leq\sqrt{c_2 n_1n_2n_3N^3}\|\mathcal{L}^{\tilde{k}}-\mathcal{L}^{\tilde{k}-1}\|_F,
\end{align*}
it immediately follows from Corollary~\ref{Cor:L} that $\lim_{\tilde{k}\to\infty}\|[\mathcal{O}^{\tilde{k}}]-[\tilde{\mathcal{O}}^{\tilde{k}}]\|_F=0$.

According to the definition of limiting subdifferential and the fact that $\varphi$ and $\psi$ are continuous functions, we can take the limits of the relations above as  $\tilde{k}\in \mathcal{K}$ approaches $\infty$. Note that $\mathcal{Z}^{\tilde{k}}\to \bar{\mathcal{Z}}$, $[\tilde{\mathcal{O}}^{\tilde{k}}]\to[\bar{\mathcal{O}}]$  and $[\mathcal{Y}^{\tilde{k}}]\to [\bar{\mathcal{Y}}]$, as $\tilde{k}\in\mathcal{K}$ approaches $\infty$. Together using Theorem~\ref{Thm:IterateSeq} (iii) and Corollary~\ref{Cor:L}, we have \eqref{Eq:optimala}, \eqref{Eq:optimale} and \eqref{Eq:optimalf} hold.

Next, we show  \eqref{Eq:optimalb}-\eqref{Eq:optimald}  hold.
Using \eqref{eq:H_ineq1} and the inequalities in the proof for assertion (i) of Theorem~\ref{Thm:IterateSeq}, we have
\begin{equation*}
    \Phi(\mathcal{Z}^k)-\Phi(\mathcal{Z}^{k+1})
\ge c_3\sum_{i=1}^{3}\left\| ([\boldsymbol{I}_{m_{i}}]-[\boldsymbol{X}_i^{k}]{[\boldsymbol{X}_i^{k}]}^{\top})[\boldsymbol{H}_i^k]\right\|_F^2
\end{equation*}
and further obtain
 \begin{align*}
 \sum^{\infty}_{k=0}\sum_{i=1}^{3}\left\| ([\boldsymbol{I}_{m_{i}}]-[\boldsymbol{X}_i^{k}]{[\boldsymbol{X}_i^{k}]}^{\top})[\boldsymbol{H}_i^k]\right\|_F^2
 \le \frac{1}{c_3}\Phi(\mathcal{Z}^0).
\end{align*}
This implies
\begin{equation}\label{eq:H1}
    \lim_{k\to \infty} \|([\boldsymbol{I}_{m_{i}}]-[\boldsymbol{X}_i^{k}]{[\boldsymbol{X}_i^{k}]}^{\top})[\boldsymbol{H}_i^k]\|=0.
\end{equation} And by taking the limit of both sides of \eqref{eq:H_ineq2} and using Theorem~\ref{Thm:IterateSeq} (iii),  we have
\begin{equation}\label{eq:H2}
    \lim_{k\to \infty}\left\|[\boldsymbol{H}_i^k]^{\top}[\boldsymbol{X}_i^{k}]-{[\boldsymbol{X}_i^{k}]}^{\top}[\boldsymbol{H}_i^k]\right\|_F=0.
\end{equation}

Since $\mathcal{Z}^{\tilde{k}}\to \bar{\mathcal{Z}}$ as $\tilde{k}\in \mathcal{K}$ approaches $\infty$, we have $[\boldsymbol{P}_{i}^{\tilde{k}}]\to [\bar{\boldsymbol{P}}_{i}]$, $[\boldsymbol{Q}_1^{\tilde{k}}]\to[\bar{\boldsymbol{Q}}_{1}]$, as $\tilde{k}\in \mathcal{K}$ approaches $\infty$. Also, we have
\begin{align*}
    \|[\boldsymbol{Q}_2^{\tilde{k}}]-[\bar{\boldsymbol{Q}}_{2}]\|_F&\leq \|[\boldsymbol{Q}_2^{\tilde{k}}]-[\tilde{\boldsymbol{Q}}_2^{\tilde{k}}]\|_F+\| [\tilde{\boldsymbol{Q}}_2^{\tilde{k}}]-[\bar{\boldsymbol{Q}}_{2}]\|_F\\
    &\leq c_1 \|[\boldsymbol{X}_1^{\tilde{k}+1}]-[\boldsymbol{X}_1^{\tilde{k}}]\|_F+\| [\tilde{\boldsymbol{Q}}_2^{\tilde{k}}]-[\bar{\boldsymbol{Q}}_{2}]\|_F
\end{align*}
and
\begin{align*}
    \|[\boldsymbol{Q}_3^{\tilde{k}}]-[\bar{\boldsymbol{Q}}_{3}]\|_F&\leq \|[\boldsymbol{Q}_3^{\tilde{k}}]-[\tilde{\boldsymbol{Q}}_{3}^{\tilde{k}}]\|_F+\|[\tilde{\boldsymbol{Q}}_{3}^{\tilde{k}}]-[\bar{\boldsymbol{Q}}_3]\|_F\\
    &\leq c_1 \|[\boldsymbol{X}_1^{\tilde{k}+1}]-[\boldsymbol{X}_1^{\tilde{k}}]\|_F+c_1 \|[\boldsymbol{X}_2^{\tilde{k}+1}]-[\boldsymbol{X}_2^{\tilde{k}}]\|_F+\| [\tilde{\boldsymbol{Q}}_3^{\tilde{k}}]-[\bar{\boldsymbol{Q}}_{3}]\|_F,
\end{align*}
where $[\tilde{\boldsymbol{Q}}_2^{\tilde{k}}]=\big([\mathcal{G}^{\tilde{k}}]\times_1 [\boldsymbol{X}_1^{\tilde{k}}]\times_3 [\boldsymbol{X}^{\tilde{k}}_3]\big)_{[(2)]}$ and $[\tilde{\boldsymbol{Q}}_{3}^{\tilde{k}}]= \left([\mathcal{G}^{\tilde{k}}]\times_1 [\boldsymbol{X}^{\tilde{k}}_1]\times_2 [\boldsymbol{X}^{\tilde{k}}_2]\right)_{[(3)]}$. Since $[\tilde{\boldsymbol{Q}}_{i}^{\tilde{k}}]\to[\bar{\boldsymbol{Q}}_{i}]$, $i=2,3$, by Theorem~\ref{Thm:IterateSeq} (iii), we have $[\boldsymbol{Q}_i^{\tilde{k}}]\to[\bar{\boldsymbol{Q}}_{i}]$, $i=2,3$, as $\tilde{k}\in \mathcal{K}$ approaches $\infty$. Hence, $[\boldsymbol{H}_i^{\tilde{k}}]\to[\bar{\boldsymbol{H}}_{i}]$, $i=1,2,3$, as ${k}\in \mathcal{K}$ approaches $\infty$. By \eqref{eq:H1} and \eqref{eq:H2}, we have \eqref{Eq:optimalb} and \eqref{Eq:optimalc} hold.
Since Stiefel manifold is a compact set, we also have \eqref{Eq:optimald} holds.
\end{proof}

\section{Application to HSI Denoising and Destriping} \label{sec:application}

In this section, we present an application of the proposed model in~\eqref{model:HSI} to HSI denoising and destriping and utilize the proposed P-BCD method given in Algorithm~\ref{alg1} for solving the model.

To denoise the HSI via the low-rank tensor regularization, we extract global, local, and nonlocal tensors that may have low-rank features via the block extraction operator $\mathscr{R}$ and impose the low-rankness on the tensors. In the following, we first introduce the block extraction for global, local, and nonlocal tensors, i.e., the choices of $\mathscr{R}$, and the sparsity enhancement on low-rank regularization for removing Gaussian white noise, i.e., the choice of $\psi$. Then we introduce the tensor $\ell_{2,p}$ norm for removing sparse noise with linear structures, i.e., the choice of $\varphi$. Lastly, we propose the MLTL2p method using a multi-scale low-rank tensor regularization with a two-phase scheme to increase the effectiveness and efficiency of model~\eqref{model:HSI} and Algorithm~\ref{alg1}.

\subsection{Block extraction}
According to the spectral correlation and spatial nonlocal self-similarity of HSIs, a clean HSI can be approximated using global, local, and nonlocal low-rank tensors~\cite{LiuMultiplicative2020,maggioni2012nonlocal}. These tensors differ in their extraction operators $\mathscr{R}$, which we later denote as $\mathscr{R}_g$ for global, $\mathscr{R}_l$ for local, and $\mathscr{R}_{nl}$ for nonlocal.

We begin by defining several types of block extraction operators 
%\(\mathscr{R}:\mathbb{R}^{I_{1}\times I_{2}\times I_{3}} \to \mathbb{R}^{m_{1}\times m_{2}\times m_{3}\times N}\) $\mathscr{R}^{\top}$
and their transpose operators. Given an HSI $\mathcal{L}\in\mathbb{R}^{I_{1}\times I_{2}\times I_{3}}$, we can divide it into a total number of $N$ overlapping or non-overlapping blocks of size $r_1\times r_2\times r_3$, $r_i\leq I_i$. We define the operator that extracts the $l$-th block as $\mathscr{B}_l:\mathbb{R}^{I_1\times I_2\times I_3}\to \mathbb{R}^{r_1\times r_2\times r_3}$
\begin{displaymath}
    \mathscr{B}_l(\mathcal{L}):=\mathcal{L}\times_1 (\boldsymbol{U}_{1}^{(l)})^{\top}\times_2 (\boldsymbol{U}_{2}^{(l)})^{\top}\times_3 (\boldsymbol{U}_{3}^{(l)})^{\top},
\end{displaymath}
where $\boldsymbol{U}_{i}^{(l)}\in\mathbb{R}^{I_i\times r_i}$ is a binary matrix such that $\mathscr{B}_l(\mathcal{L})$ is exactly the $l$-th local block of $\mathcal{L}$, and $(\boldsymbol{U}_{i}^{(l)})^{\top}\boldsymbol{U}_{i}^{(l)}=\mathbf{I}_{r_i}$. The transpose of $\mathscr{B}_l$ can be defined by $\mathscr{B}_l^{\top}:\mathbb{R}^{r_1\times r_2\times r_3}\to\mathbb{R}^{I_1\times I_2\times I_3}$
\begin{displaymath}
    \mathscr{B}_l^{\top}(\mathcal{Y}):=\mathcal{Y}\times_1 \boldsymbol{U}_{1}^{(l)}\times_2 \boldsymbol{U}_{2}^{(l)}\times_3 \boldsymbol{U}_{3}^{(l)},
\end{displaymath}
which satisfies 
\begin{displaymath}
	\langle \mathcal{Y},\mathscr{B}_l(\mathcal{L})\rangle_F=\langle \mathscr{B}^{\top}_l(\mathcal{Y}),\mathcal{L}\rangle_F,
\end{displaymath}
for any $\mathcal{L}\in \mathbb{R}^{I_1\times I_2\times I_3}$ and $\mathcal{Y}\in \mathbb{R}^{r_1\times r_2\times r_3}$.
%such that $\mathcal{L}\times_1 U_{1}^{(l)}\times_2 U_{2}^{(l)}$ is exactly the $l$-th local block of $\mathcal{L}$. 
For those $N$ independent blocks, we can stack them together and form a new operator
$\mathscr{B}:\mathbb{R}^{I_1\times I_2\times I_3}\to \mathbb{R}^{r_1\times r_2\times r_3\times N}$ such that $[\mathscr{B}(\mathcal{L})]^{(l)}=\mathscr{B}_l(\mathcal{L})$.

Depending on the size of the blocks, there are several types of block extraction as follows using the formulation of $\mathscr{B}$: 
\begin{itemize}
    \item {\bf Global extraction $\mathscr{R}_g$:} $\mathscr{B}$ with $r_i=I_i$, $i=1,2,3$;
    \item {\bf Local extraction $\mathscr{R}_l$:} $\mathscr{B}$ with $r_i<I_i$, $i=1,2,3$;
    \item {\bf Full band extraction $\mathscr{R}_f$:} $\mathscr{B}$ with $r_i<I_i$, $i=1,2$ and $r_3=I_3$.

\end{itemize}
The tensors formed by the above block extraction operators are based on the spectral correlation of HSIs. 

Next, we formulate another type of block extraction operators, denoted as $\mathscr{R}_{nl}$, based on spatial nonlocal self-similarity of HSIs. After dividing the HSI into $N$ blocks using $\mathscr{B}_l$, we apply block matching to find similar blocks and then stack them into a fourth order nonlocal similar tensor. 

Taking FBBs as an example, for the $j$-th FBB, we search within a local window for a total of $m_2$ FBBs that are mostly similar to the reference FBB based on Euclidean distance. Then the $j$-th nonlocal similar sub-tensor of order 3 of \(\mathcal{L}\), denoted as $\mathscr{R}_j(\mathcal{L})$, can be formed by unfolding all the nonlocal similar blocks in the $j$-th group and then stacking them together.  
We define the Casorati matrix (a matrix
whose columns are vectorized bands of the HSI) of the $l$-th FBB as follows
\begin{displaymath}
    \tilde{\boldsymbol{B}}_l:=\operatorname{reshape}(\mathscr{B}_l(\mathcal{L}),m_1,m_3),
\end{displaymath}
where  $m_1=r_1r_2$ and $m_3=I_3$.
Then the extraction operator of  the $j$-th nonlocal similar sub-tensor $\mathscr{R}_j:\mathbb{R}^{I_1\times I_2\times I_3}\to \mathbb{R}^{m_1\times m_2\times m_3}$ can be defined by
\begin{displaymath}
    \mathscr{R}_j(\mathcal{L}):=\operatorname{reshape}([\tilde{\boldsymbol{B}}_{l_{j1}}^{\top},\tilde{\boldsymbol{B}}_{l_{j2}}^{\top},\dots, \tilde{\boldsymbol{B}}_{l_{jm_2}}^{\top}]^{\top},m_1,m_2,m_3),
\end{displaymath}
where the indices $l_{j1}, l_{j1}, \dots, l_{jm_2}$ refer to the indices of FBBs that belongs to the $j$-th nonlocal similar group and $j=1,2,\dots, N$. 
An illustrative figure on nonlocal low-rank tensor extraction is shown in Fig.~\ref{Nonlocalregularization}. Then the transpose $\mathscr{R}^{\top}_j:\mathbb{R}^{m_1\times m_2\times m_3}\to \mathbb{R}^{I_1\times I_2\times I_3}$ is defined by
\begin{equation}\label{eq:RjT}
\mathscr{R}^{\top}_j(\mathcal{Y}):=\sum_{i=1}^{m_2}\mathscr{B}_{l_{ji}}^{\top}\left(\operatorname{reshape}(\boldsymbol{Y}_{:i:},r_1,r_2,I_3)\right),
\end{equation}
which satisfies $\langle \mathcal{Y},\mathscr{R}_j(\mathcal{L})\rangle_F=\sum_{i=1}^{m_2}\langle {Y}_{:i:},\mathscr{B}_{l_{ji}}(\mathcal{L})\rangle_F=\langle \mathscr{R}^{\top}_j(\mathcal{Y}),\mathcal{L}\rangle_F$.
Stacking $\mathscr{R}_j(\mathcal{L})$ into a fourth order nonlocal similar group tensor, the extraction operator $\mathscr{R}_{nl}:\mathbb{R}^{I_1\times I_2\times I_3}\to \mathbb{R}^{m_1\times m_2\times m_3\times N}$ is defined as a linear map such that $[\mathscr{R}_{nl}(\mathcal{L})]^{(j)}=\mathscr{R}_j(\mathcal{L})$.
\begin{figure}[h]
\centering
\includegraphics[width=\textwidth]{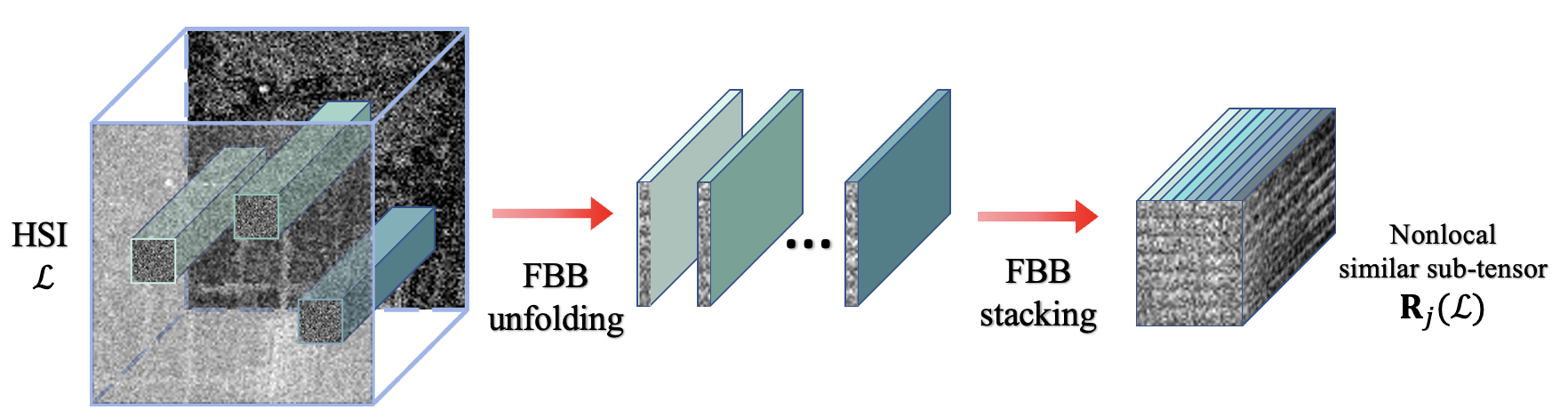} %
\caption{The procedure of block matching.}
\label{Nonlocalregularization}
\end{figure}

Note that the $\mathscr{R}$ formulated above all satisfy the following equality
\begin{align*}
\mathcal{W}_{\mathscr{R}}\odot \operatorname{Id}=\mathscr{R}^{\top}\mathscr{R},
\end{align*}
where $\odot$ represents the pointwise multiplication, and each entry of $\mathcal{W}_{\mathscr{R}}\in\mathbb{R}^{I_{1}\times I_{2}\times I_{3}}$ represents the number of blocks to which the corresponding pixel belongs. For $\mathscr{R}_{g}$, $\mathscr{R}_{l}$ and $\mathscr{R}_{f}$, the weight tensor can be defined by
\begin{displaymath}
\mathcal{W}_{\mathscr{R}}\odot \operatorname{Id}=\sum_{l=1}^N\mathscr{B}_l^{\top}\mathscr{B}_l,
\end{displaymath}
and for $\mathscr{R}_{nl}$, the weight tensor can be defined by
\begin{displaymath}
\mathcal{W}_{\mathscr{R}}\odot \operatorname{Id}=\sum_{j=1}^N\mathscr{R}_j^{\top}\mathscr{R}_j.
\end{displaymath}
 Since we assume that each pixel belongs to at least one low-rank group, we have $\mathcal{W}_{\mathscr{R}}\in\mathbb{R}^{I_1\times I_2\times I_3}_{++}$.

\subsection{Sparsity enhanced low-rank tensor regularization}

Taking $\mathscr{R}_{nl}$ as an example, in the $\mathscr{R}_{nl}(\mathcal{L})$ that we construct, the first dimension indicates the spatial information, the second dimension reveals the nonlocal self-similarity, the third dimension reflects the spectral correlation, and the fourth dimension shows no correlation. Hence, we adopt the independent 3-D HOSVD to obtain a low-rank approximation of $\mathscr{R}(\mathcal{L})$, that is,
\begin{displaymath}
    \mathscr{R}(\mathcal{L})\approx\mathcal[\mathcal{G}]\times_1 [\boldsymbol{X}_1]\times_2 [\boldsymbol{X}_2]\times_3 [\boldsymbol{X}_3],
\end{displaymath}
where $[\mathcal{G}]\in \mathbb{R}^{n_1\times n_2\times n_3\times N}$ denotes independent core tensors, and $[\boldsymbol{X}_i]\in\mathbb{R}^{m_{i}\times n_{i}\times N}$ denotes the $i$-th factor matrices such that $[\boldsymbol{X}_i]^{\top}[\boldsymbol{X}_i]=[\boldsymbol{I}_{n_i}]$.

To further boost the low-rankness of $\mathscr{R}(\mathcal{L})$, in additional to restricting the core tensors to a unified size,  we propose a sparsity-enhanced nonlocal low-rank tensor regularization term  as follows
\begin{align}\label{Eq:SparseEnhance}
\frac{1}{2}\|\mathcal[\mathcal{G}]\times_1 [\boldsymbol{X}_1]\times_2 [\boldsymbol{X}_2]\times_3 [\boldsymbol{X}_3]-\mathscr{R}(\mathcal{L})\|^{2}_{F}+\|[\mathcal{G}]\|_{1,\boldsymbol{w}},
\end{align}
where $\|[\mathcal{G}]\|_{1,\boldsymbol{w}}$ is given in~\eqref{eq:Weightedl1Norm} with $w\in\mathbb{R}^N_{++}$.
In particular, the first term of~\eqref{Eq:SparseEnhance} measures the closeness between $\mathscr{R}(\mathcal{L})$ and the approximated low-rank tensor, and the second term measures the sparsity of the independent core tensors $[\mathcal{G}]$. Using the proximal operator of the $\ell_1$ norm, the proximal operator of $\|\cdot\|_{1,\boldsymbol{w}}$ can be computed component-wisely by the soft thresholding operator as follows
 \begin{align*}
\left(\operatorname{prox}_{\|\cdot\|_{1,\boldsymbol{w}}}([\mathcal{G}])\right)_{i_1i_2i_3}^{(j)}& = \operatorname{prox}_{w_j |\cdot|}\left([\mathcal{G}]^{(j)}_{i_1i_2i_3}\right)\\
&=\operatorname{sign}\left([\mathcal{G}]^{(j)}_{i_1i_2i_3}\right)\max\left(\left|[\mathcal{G}]^{(j)}_{i_1i_2i_3}\right|-w_j,0\right).
\end{align*}

\subsection{Tensor $\ell_{2,p}$ norm for group sparsity} 

To remove stripes in HSIs, we measure the linear structural sparsity of the sparse noise tensor $\mathcal{S}$ using the tensor $\ell_{2,p}$ norm. The tensor $\ell_{2,p}$ norm for group sparsity is extended from the matrix $\ell_{2,p}$ norm, which has been applied to image processing~\cite{li2017robust},  machine learning~\cite{ding2021laplacian,ma2017l2p}, feature selection~\cite{li2021sparse,wang2013l_}, multi-view classification~\cite{wang2022block}, etc.
As the stripes and dead lines often align the first dimension, we define the tensor  $\ell_{2,p}$ $(0<p<1)$ norm in the form of~\eqref{eq:L2pNorm}, which is exactly the matrix $\ell_{2,p}$ norm of the unfolding matrix along the first dimension.

In the following, we summarize some results for solving the tensor $\ell_{2,p}$ norm minimization problem
 \begin{equation}\label{eq:L2pmin}
	\min_{\mathcal{S}}\mu\|\mathcal{S}\|^{p}_{2,p}+\frac{1}{2}\|\mathcal{S}-\widetilde{\mathcal{S}}\|^{2}_{F},
\end{equation}
 where $\widetilde{\mathcal{S}}\in \mathbb{R}^{I_1\times I_2\times I_3}$ is a given tensor and parameter $\mu>0$.
Since $\|\cdot\|^{p}_{2,p}$ and $\|\cdot\|^{2}_{F}$ are both group-separable, solving problem \eqref{eq:L2pmin} for $\mathcal{S}$ is equivalent to solving the following subproblem for each $(i_2,i_3)$-th vector of $\mathcal{S}$ along the first dimension
 \begin{equation}\label{L2p}
	\min_{\boldsymbol{s}}\mu\|\boldsymbol{s}\|^{p}_{2}+\frac{1}{2}\|\boldsymbol{s}-\tilde{\boldsymbol{s}}\|^{2}_{2},
\end{equation}
where $\boldsymbol{s}\in \mathbb{R}^{I_1}$ and $\tilde{\boldsymbol{s}}\in \mathbb{R}^{I_1}$, for simplicity, represent $\boldsymbol{s}_{:i_2i_3}$ and $\tilde{\boldsymbol{s}}_{:i_2i_3}$, respectively, and $\|\boldsymbol{s}\|^p_2=({s}_{1}^{2}+{s}_{2}^{2}+\cdots+{s}_{I_1}^{2})^{\frac{p}{2}}$. It follows from the triangle inequality that the objective function of \eqref{L2p} satisfies the following inequality for any $s\in \mathbb{R}^{I_1}$
\begin{equation}\label{ineq:s2p}
\mu\|\boldsymbol{s}\|^{p}_{2}+\frac{1}{2}\|\boldsymbol{s}-\tilde{\boldsymbol{s}}\|^{2}_{2}\geq \mu\|\boldsymbol{s}\|^{p}_{2}+\frac{1}{2}(\|\boldsymbol{s}\|_2-\|\tilde{\boldsymbol{s}}\|_2)^{2}.
 \end{equation}
And the equality holds if and only if $\boldsymbol{s}=t\tilde{\boldsymbol{s}}$ for some $t\geq 0$ or $\tilde{\boldsymbol{s}}=0$. Observe that the right-hand side of the inequality is only related to $\|\boldsymbol{s}\|_2$ and $\|\tilde{\boldsymbol{s}}\|_2$. If $\tilde{\boldsymbol{s}}=0$, the  solution of problem \eqref{L2p} is $\boldsymbol{s}=\boldsymbol{0}$. If $\tilde{\boldsymbol{s}}\ne \boldsymbol{0}$, we can view $s$ as $s=t\|\tilde{\boldsymbol{s}}\|_2\boldsymbol{v}$ with $t\geq 0$ being a scalar and $v\in\mathbb{R}^{I_1}$ being a unit vector. When we restrict the minimization problem \eqref{L2p} by $\|\boldsymbol{s}\|_2=t\|\tilde{\boldsymbol{s}}\|_2$ with a fixed $t$, according to~\eqref{ineq:s2p}, the solution of the restricted problem of \eqref{L2p} is obtained only when $\boldsymbol{v}=\frac{\tilde{\boldsymbol{s}}}{\|\tilde{\boldsymbol{s}}\|_2}$. Hence, if $\tilde{\boldsymbol{s}}\ne \boldsymbol{0}$, the solution of \eqref{L2p} is $\boldsymbol{s}=t\tilde{\boldsymbol{s}}$, where $t$ is a minimizer of the following problem
\begin{equation}\label{t}
\min_{t\in[0,\infty)} \nu t^{p}+\frac{1}{2}(t-1)^2,
\end{equation}
with $\nu=\mu\|\tilde{\boldsymbol{s}}\|_2^{p-2}$. That is, it only requires to solve a one-dimensional problem~\eqref{t} for computing the solutions of problem~\eqref{L2p}.

Next, we show a lemma and a solver for computing solutions of problem~\eqref{L2p}. Let $g(t)=\nu |t|^{p}+\frac{1}{2}(t-1)^2$. Note that $g(t)>g(|t|)$ for $\forall t<0$ and $g(t)>g(1)$ for $\forall t>1$. Then problem~\eqref{t} can be relaxed to an unconstrained problem with $g$ being the objective function, which can be solved using Theorem 1 in \cite{6262492}. Also, problem~\eqref{t} can be reduced to a box constrained problem with constraint $t\in [0,1]$, which can be solved using Lemma 4.1 in \cite{10Chen069}. We summarize the results for ~\eqref{t} in the following lemma.
\begin{lemma}\label{lem:proxl2p} Let $p\in(0,1)$ and $\nu>0$. Let
\begin{displaymath}
\nu_{0}:=\frac{[2 (1-p)]^{1-p}}{(2-p)^{2-p}}, \quad\text{ and }\quad
 \tau(\nu):=[2\nu(1-p)]^{\frac{1}{2-p}}.
\end{displaymath}
Then the set of optimal solutions of problem (\ref{t}), denoted as  $\Omega^{*}(\nu)$, is given by
\begin{displaymath}
    \Omega^{*}(\nu)=\begin{cases}
        \{0\},\quad&\mbox{if }\nu >\nu_0;\\
        \{0,\tau(\nu_0)\},\quad &\mbox{if }\nu=\nu_0;\\
        \{t^*\},\quad&\mbox{if }0<\nu<\nu_0,
    \end{cases}
\end{displaymath}
where $t^*\in(\tau(\nu),1)$ is the unique solution of the equation
\begin{equation}
 \label{roott}
    \nu pt^{p-1}+t-1=0
\end{equation}
with $t\in (\tau(\nu),\infty)$.
\end{lemma}

According to Lemma~\ref{lem:proxl2p}, when $\nu=\nu_0$, there are two minimizers for problem~\eqref{t}. For simplicity, we will choose $0$ in this case. When $\nu \in (0,\nu_0)$, the minimizer of problem~\eqref{t} is unique and can obtained by solving \eqref{roott}. If $p$ is chosen as, for example, $p=1/2$, \eqref{roott} has a closed-form root. Otherwise, we estimate the unique root $t^{*}$ by Newton’s method with an initial value of $t_0=(\tau(\nu) + 1)/2$. Altogether, we summarize a proximal operator of the tensor $\ell_{2,p}$ norm in the following theorem.

\begin{theorem}\label{thm:proxl2p} Let $p\in (0,1)$ and $\mu>0$. Define the operator $\Gamma_{\mu}:\mathbb{R}\to \mathbb{R}$ by
	\begin{displaymath}
		\Gamma_{\mu}(\beta):=\begin{cases}0,\quad&\text{if }\beta\leq \frac{\beta_0(2-p)}{2(1-p)};\\
			t^*,\quad&\text{otherwise},\end{cases}
	\end{displaymath}
	where $\beta_0=[2\mu(1-p)]^{\frac{1}{2-p}}$, and $t^*\in\left[\beta_0\beta,1\right)$ is the unique solution of
	\begin{equation*}
		\mu \beta^{p-2}pt^{p-1}+t-1=0,\quad  t\in\left[\beta_0\beta,\infty\right).
	\end{equation*}
	Then a solution of the proximal operator of the tensor $\ell_{2,p}$ norm at $\widetilde{\mathcal{S}}\in\mathbb{R}^{I_1\times I_2\times I_3}$ can be computed by
	\begin{displaymath}
		\Gamma_{\mu}(\|\tilde{\boldsymbol{s}}_{:i_2i_3}\|_2)\tilde{\boldsymbol{s}}_{:i_2i_3}\in \operatorname{prox}_{\mu\|\cdot\|_2^p}(\tilde{\boldsymbol{s}}_{:i_2i_3})=\left(\operatorname{prox}_{\mu\|\cdot\|_{2,p}^p}(\widetilde{\mathcal{S}})\right)_{:i_2i_3},
	\end{displaymath}
	for $i_2=1,2,\dots,I_2$, $i_3=1,2,\dots,I_3$.
	
\end{theorem}

By applying Theorem~\ref{thm:proxl2p}, the update of $\mathcal{S}$ given in~\eqref{eq:Supdate} can be efficiently computed. In particular, we calculate the $(i_2,i_3)$-th mode-$1$ fiber of $\mathcal{S}^{k+1}$ by
	\begin{displaymath}
		\boldsymbol{s}_{:i_2i_3}^{k+1}=\Gamma_{\tilde{\gamma}}\left(\left\| \tilde{\boldsymbol{s}}^{k}_{:i_2i_3}\right\|_2\right)\tilde{\boldsymbol{s}}^{k}_{:i_2i_3},
	\end{displaymath}
where  $\tilde{\mathcal{S}}^k= \mathcal{S}^k-\tilde{\alpha}_{\mathcal{S}}\left(\mathcal{S}^k+\mathcal{L}^k- \mathcal{D}\right)$, and $\tilde{\boldsymbol{s}}^{k}_{:i_2i_3}$ is the $(i_2,i_3)$-th mode-$1$ fibers of $\tilde{\mathcal{S}}^k$.

\subsection{Two-phase MLTL2p method}
We extend model~\eqref{model:HSI} to a multi-scale low-rank model and apply the P-BCD method in a two-phase fashion for HSI denoising and destriping.

The spectral correlation and spatial nonlocal self-similarity in HSIs imply multi-scale low-rankness, which includes global low-rankness, local block low-rankness, and nonlocal self-similar tensor low-rankness. In the following, we propose a multi-scale low-rank model that imposes low-rankness on global, local, and nonlocal tensors extracted via $\mathscr{R}_g$, $\mathscr{R}_l$, and $\mathscr{R}_{nl}$, respectively.
\begin{equation}\label{model:HSI_new}
\begin{split}
\min_{\substack{[\boldsymbol{X}_i]_g,[\boldsymbol{X}_i]_l,[\boldsymbol{X}_i]_{nl},\\
[\mathcal{G}]_g,[\mathcal{G}]_l,[\mathcal{G}]_{nl},\mathcal{S},\mathcal{L}} } \quad
 &\frac{1}{2} \|\mathcal{L}+\mathcal{S}-\mathcal{D}\|_{F}^2+\gamma \|\mathcal{S}\|^{p}_{2,p}\\
+&\|[\mathcal{G}]_g\|_{1,\boldsymbol{w}}
+ \frac{\delta_g}{2}\|\mathscr{R}_g(\mathcal{L})- [\mathcal{G}]_g\times
_{1}[\boldsymbol{X}_1]_g\times_{2}[\boldsymbol{X}_2]_g\times_{3}[\boldsymbol{X}_3]_g\|^{2}_{F}\\
+&\|[\mathcal{G}]_l\|_{1,\boldsymbol{w}}
+ \frac{\delta_l}{2}\|\mathscr{R}_l(\mathcal{L})- [\mathcal{G}]_l\times
_{1}[\boldsymbol{X}_1]_l\times_{2}[\boldsymbol{X}_2]_l\times_{3}[\boldsymbol{X}_3]_l\|^{2}_{F}\\
+&\|[\mathcal{G}]_{nl}\|_{1,\boldsymbol{w}}
+ \frac{\delta_{nl}}{2}\|\mathscr{R}_{nl}(\mathcal{L})- [\mathcal{G}]_{nl}\times
_{1}[\boldsymbol{X}_1]_{nl}\times_{2}[\boldsymbol{X}_2]_{nl}\times_{3}[\boldsymbol{X}_3]_{nl}\|^{2}_{F}\\
\mbox{s.t. }~~& [\boldsymbol{X}_i]_g^{\top}[\boldsymbol{X}_i]_g=[\boldsymbol{I}_{n_i}], \quad i=1,2,3,\\
& [\boldsymbol{X}_i]_l^{\top}[\boldsymbol{X}_i]_l=[\mathbf{I}_{n_i}], \quad i=1,2,3,\\
& [\boldsymbol{X}_i]_{nl}^{\top}[\boldsymbol{X}_i]_{nl}=[\boldsymbol{I}_{n_i}], \quad i=1,2,3.
\end{split}
\end{equation}
And the P-BCD method for solving model~\eqref{model:HSI_new} is presented in Algorithm~\ref{alg2}.
\begin{algorithm}[h]
		\renewcommand{\algorithmicensure}{\textbf{Output:}}
	\caption{P-BCD algorithm for model~\eqref{model:HSI_new}}
	\label{alg2}
	\begin{algorithmic}[1]
		\STATE Initialize $(\mathcal{S}^0,[\boldsymbol{X}_i^0]_g,[\boldsymbol{X}_i^0]_l,[\boldsymbol{X}_{i}^0]_{nl},[\mathcal{G}^0]_g,[\mathcal{G}^0]_l,[\mathcal{G}^0]_{nl},\mathcal{L}^0)$ 
        
        with \([\boldsymbol{X}_i^0]_g,[\boldsymbol{X}_i^0]_l,[\boldsymbol{X}_{i}^0]_{nl}\) being independently orthogonal;
  \STATE Set the tensor extraction operators $\mathscr{R}_g$, $\mathscr{R}_l$, and $\mathscr{R}_{nl}$;
		\STATE  Set parameters $\alpha_{\mathcal{S}}, \alpha_{\boldsymbol{X}}, \alpha_{\mathcal{G}}>0$;
  \STATE Set $k=0$.
		\REPEAT
		\STATE Compute $\mathcal{S}^{k+1}$ by \eqref{eq:Supdate};
  \STATE Compute $[\boldsymbol{X}_i^{k+1}]_g,[\boldsymbol{X}_i^{k+1}]_l,[\boldsymbol{X}_{i}^{k+1}]_{nl}$ by \eqref{eq:Xupdate},  $i=1,2,3$;
  \STATE Compute $[\mathcal{G}^{k+1}]_g,[\mathcal{G}^{k+1}]_l,[\mathcal{G}^{k+1}]_{nl}$ by \eqref{eq:Gupdate};
\STATE Compute $\mathcal{L}^{k+1}$ by 
\begin{align*}
    \mathcal{L}^{k+1}= &\left(\delta_g\mathcal{W}_{\mathscr{R}_g}+\delta_l\mathcal{W}_{\mathscr{R}_l}+\delta_{nl}\mathcal{W}_{\mathscr{R}_{nl}}+\mathcal{I}\right)^{-1}\\
    &\odot\Big(\delta_g\mathscr{R}_g^{\top}([\mathcal{Y}^{k+1}]_g) +\delta_l\mathscr{R}_l^{\top}([\mathcal{Y}^{k+1}]_l)\\
    &+\delta_{nl}\mathscr{R}_{nl}^{\top}([\mathcal{Y}^{k+1}]_{nl}) +(\mathcal{D}-\mathcal{S}^{k+1})\Big);
\end{align*}
		\STATE $k\leftarrow k+1$.
		\UNTIL the stopping criterion is met.
		\ENSURE $(\mathcal{S}^k,[\boldsymbol{X}_i^k]_g,[\boldsymbol{X}_i^l]_l,[\boldsymbol{X}_{i}^k]_{nl},[\mathcal{G}^k]_g,[\mathcal{G}^k]_l,[\mathcal{G}^k]_{nl},\mathcal{L}^k)$ . 
        \end{algorithmic}
\end{algorithm}

To enhance the efficiency and effectiveness of the MLTL2p method, we implement a two-phase scheme on the multi-scale low-rank model and the P-BCD method. In the first phase, we employ only the global and local low-rank regularization, which is sufficient for identifying stripes and dead lines. In particular, we use Algorithm~\ref{alg2} to solve model~\eqref{model:HSI_new} with the parameter $\delta_{nl}$ set to zero. In the second phase, we utilize the nonlocal low-rank regularization along with the global and local low-rank regularization. As the estimated HSI from the first phase has already removed most of the noise, we can  perform only one block matching to obtain a fixed nonlocal block extraction $\mathscr{R}_{nl}$, which can significantly reduce the computational cost. By applying Algorithm~\ref{alg2} to solve model~\eqref{model:HSI_new} in this phase, the algorithm can refine additional details in the HSI, thereby improving the quality of the restored HSI.

Lastly, we comment on some situations in HSI denoising and destriping. If Poisson noise exists in HSIs, one may conduct the Anscombe transformation~\cite{makitalo2010optimal} and then apply the MLTL2p method. Also, if the HSIs are severely contaminated by the dead lines or stripes, one may apply the median method or the midway equalization method~\cite{tendero2012non}, respectively, for preprocessing.

% % -------------------- 5. Experiment Section -------------------- % %

\section{Experimental Results}\label{sec:experiments}

In this section, we conduct numerical experiments for removing mixed noise in HSIs. We summarize the comparison of methods for HSI denoising and destriping in Table~\ref{table_new}.  Among them, we select the methods that utilize tensor representation for competing. These include BM4D~\cite{maggioni2012nonlocal} with FGLR~\cite{su2023fast} as a preprocessing step to remove stripes,  as well as LRTD~\cite{chen2018destriping}, SNLRSF~\cite{cao2019hyperspectral}, LRTFL0~\cite{xiong2019hyperspectral},  and the data-driven method QRNN3D~\cite{wei20203}. All numerical experiments are implemented in Matlab R2018a and executed on a personal desktop (Intel Core i7 9750H at 2.60 GHz with 16 GB RAM).

{\small\begin{table}[h]
	\centering
	\setlength{\tabcolsep}{4pt}
	\caption{Comparison summary of the methods for HSI denoising and destriping.}\label{table_new}
	\begin{tabular}{l|c|c|c|c|l|l}
		\hline 
		\multirow{2}{*}{\textbf{Methods}} & \multirow{2}{*}{\textbf{Form}} & \multicolumn{3}{c|}{\textbf{Address Correlations?}} & \multirow{2}{*}{\textbf{Stripe}} & \multirow{2}{*}{\textbf{Low-rank}}\\
		\cline{3-5}
		&  & \textbf{Global} & \textbf{Local}  & \scalebox{0.9}[1]{\textbf{Nonlocal}} & &\\
		\hline
		\textbf{BM4D}~\cite{maggioni2012nonlocal} & Tensor  & No& No & Yes & No&Hard-thresholding\\
		\textbf{FGLR}~\cite{su2023fast}  & Matrix & Yes&No & No& $\|\cdot\|_*$&Graph Laplacian\\	\textbf{LRMR}~\cite{zhang2013hyperspectral} & Matrix  & No& Yes & No & $\|\cdot\|_1$&Go Decomposition\\
		\textbf{DLR}~\cite{zhang2021double} & Matrix  & Yes& No & No & $\|\cdot\|_*$& $\|\cdot\|_*$\\
		\scalebox{0.7}[1]{\textbf{SSLR-SSTV}}\cite{yang2021hyperspectral} & Matrix  & No& Yes & No & $\|\cdot\|_*$& $\|\cdot\|_*$\\
		\textbf{LRTD}~\cite{chen2018destriping} &Tensor &Yes&No&No&$\|\cdot\|_{2,1}$& 3-D HOSVD\\
		\textbf{SNLRSF}~\cite{cao2019hyperspectral}  &Tensor&Yes &No&Yes&$\|\cdot\|_1$& Successive SVD \\
		\textbf{LRTFL0}~\cite{xiong2019hyperspectral} & Tensor &Yes&No&No&$\|\cdot\|_1$& Block Term Decomp.  \\
		\textbf{MLTL2p} & Tensor &Yes&Yes&Yes& $\|\cdot\|_{2,p}$  &Indep. 3-D HOSVD \\\hline
	\end{tabular}
\end{table}}

% % ---------- 5.1 Simulated Data Experiments

\subsection{Simulated data experiments}

In this subsection, the proposed method and the competing methods are tested on simulated data generated from the Washington DC Mall\footnote{\href{https://engineering.purdue.edu/~biehl/MultiSpec/hyperspectral.html}{https://engineering.purdue.edu/$\sim$biehl/MultiSpec/hyperspectral.html}}, Xiong-An\footnote{\href{http://www.hrs-cas.com/a/share/shujuchanpin/2019/0501/1049.html}{http://www.hrs-cas.com/a/share/shujuchanpin/2019/0501/1049.html}} and Chikusei\footnote{\href{https://naotoyokoya.com/Download.html}{https://naotoyokoya.com/Download.html}} datasets. Three test images of size $128 \times 128 \times 128$ are randomly obtained from each dataset.

To simulate the noisy HSI data, Gaussian noise, stripes, or dead lines are added to the normalized clean HSI data using the procedure in~\cite{7542167} under the following cases:
\begin{itemize}
   \item \emph{Case 1:} Gaussian noise with a mean of zero and a standard deviation of $0.1$ is added to all the bands.  Additionally, bands $45-60$ and $105-120$ (i.e., $25\%$ of the bands) are selected to impose stripe noise. 
   Stripe noise with an intensity (mean absolute value of pixels) of $0.2$ and a percentage of $10\%$ per band (i.e., $2.5\%$ of the HSI) is added to the selected bands.
   
   \item \emph{Case 2:} Gaussian noise with a mean of zero and a standard deviation of $0.1$ is added to all the bands. Additionally, all the bands are selected to impose stripe noise. Stripe noise with an intensity of $0.2$ and a percentage of $10\%$ per band (i.e., $10\%$ of the HSI) is added to all the bands.

        \item \emph{Case 3:} Gaussian noise with a mean of zero and a standard deviation of $0.1$ is added to all the bands. Additionally,  all the bands are selected to impose dead lines. Dead lines with a percentage of $5\%$ per band (i.e., $5\%$ of the HSI) are added to the selected bands.

     \item \emph{Case 4:}  Gaussian noise with a mean of zero and a standard deviation randomly varying from $0.1$ to $0.2$ is added to all the bands. Additionally, $25\%$ of the bands are selected from band $1-64$ to impose stripe noise with an intensity of $0.2$ and a percentage of $10\%$ per band; $12.5\%$ of the bands are selected from band $65-128$ to impose dead lines with a percentage of $5\%$ per band.
       
\end{itemize}

For comparing the quality of the restored images, four evaluation metrics are employed, which are the mean peak signal-to-noise ratio (MPSNR), the mean structural similarity index (MSSIM), the mean feature similarity index (MFSIM), and the erreur relative globale adimensionnelle de synthese (ERGAS). Let $\mathcal{X}^{*}$ denote the restored HSI and $\hat{\mathcal{X}}$ denote the clean HSI. Then ${\boldsymbol{X}^{*}_{::i}}$ and $\hat{\boldsymbol{X}}_{::i}$ denote the $i$-th band of the restored HSI and clean HSI, respectively. The MPSNR value is defined as
\begin{displaymath}
\mathrm{MPSNR}
=\frac{1}{I_3}\sum^{I_3}_{i=1}10\log_{10}\left(\frac{\max^2(\boldsymbol{X}^{*}_{::i})}{ \mathrm{mse}(\boldsymbol{X}^{*}_{::i},\hat{\boldsymbol{X}}_{::i})}\right),
\end{displaymath}
which is the average PSNR value across the bands. Similarly, MSSIM and MFSIM values are defined as
\begin{displaymath}
\mathrm{MSSIM}=\frac{1}{I_3}\sum^{I_3}_{i=1}\mathrm{SSIM}(\boldsymbol{X}^{*}_{::i},\hat{\boldsymbol{X}}_{::i})\quad \mbox{ and }\quad\mathrm{MFSIM}=\frac{1}{I_3}\sum^{I_3}_{i=1}\mathrm{FSIM}(\boldsymbol{X}^{*}_{::i},\hat{\boldsymbol{X}}_{::i}),
\end{displaymath}
where SSIM is given in  \cite{zhou2004image} and FSIM is given in  \cite{zhang2011fsim}. The ERGAS are defined as
\begin{displaymath}
\mathrm{ERGAS}=100\sqrt{\frac{1}{I_3}\sum^{I_3}_{i=1}\frac{\mathrm{mse}(\boldsymbol{X}^{*}_{::i},\hat{\boldsymbol{X}}_{::i})}{\mathrm{mean}(\boldsymbol{X}^{*}_{::i})}}.
\end{displaymath}
 In addition, better denoising results are indicated by larger MPSNR, MSSIM, and MFISM values, as well as smaller  ERGAS value.

For the proposed MLTL2p method, we set algorithm parameters  $\alpha_{\mathcal{S}}=0.1$, and $\alpha_{X}=\alpha_{\mathcal{G}}=0.01$ for Algorithm~\ref{alg2}. In the first phase, we set model parameters $w_j=10^{-2}$, $p=0.1$, and $\delta_{nl}=0$; we set $\gamma=0.8$ for case 1 and 2 and  $\gamma=1$ for case 3 and 4; 
for $\mathscr{R}_g$, we set $\delta_g=1$, block size $[128,128,128]$ and block rank $[102,102,3]$; for $\mathscr{R}_l$, we set $\delta_l=1$, block size $[32,32,32]$ and block rank $[26,26,2]$. Algorithm~\ref{alg2} stops at iteration $10$. In the second phase, we set model parameters $w_j=10^{-2}$, $p=0.1$;  we set $\gamma=1.76$ for case 1 and 2 and  $\gamma=2.2$ for case 3 and 4; for $\mathscr{R}_g$, we set $\delta_g=3$, block size $[128,128,128]$ and block rank $[102,102,5]$; for $\mathscr{R}_l$, we set $\delta_l=3$, block size $[32,32,32]$ and block rank $[26,26,3]$; for $\mathscr{R}_{nl}$, we set $\delta_{nl}=60/\operatorname{median}(\mathcal{W}_{\mathscr{R}})$, block size $[36,128,128]$ and block rank $[32,43,5]$. Algorithm~\ref{alg2} stops if $\|\mathcal{L}^{k+1}-\mathcal{L}^k\|_F/\|\mathcal{L}^{k+1}\|_F \leq 0.005$ and  $\|\mathcal{S}^{k+1}-\mathcal{S}^k\|_F/\|\mathcal{S}^{k+1}\|_F \leq 0.005$.
For the competing methods, the parameters are manually tuned according to the settings provided in their articles to ensure optimal performance.

\begin{table}[h]
    \centering
    \setlength{\tabcolsep}{4pt}
    \caption{Average numerical results tested on simulated data.}\label{table1}
\begin{tabular}{c|c|ccccccc}
\hline
% \scalebox{0.9}[1]{\textbf{Case}} &  \scalebox{0.9}[1]{\textbf{Index}} & \scalebox{0.9}[1]{\textbf{Noisy}}   &  \scalebox{0.9}[1]{\textbf{MLTL2p}}     &  \scalebox{0.8}[1]{\textbf{LRTFL0}}      &  \scalebox{0.8}[1]{\textbf{SNLRSF}}    &  \scalebox{0.8}[1]{\textbf{LRTD}}     &  \scalebox{0.6}[1]{\textbf{FGLR+BM4D}}   &  \scalebox{0.8}[1]{\textbf{QRNN3D}}\\ \hline 
  \multirow{2}{*}{\scalebox{0.9}[1]{\textbf{Case}}} &   \multirow{2}{*}{\scalebox{0.9}[1]{\textbf{Index}}} & \multirow{2}{*}{ \scalebox{0.9}[1]{\textbf{Noisy}}}   &  \multirow{2}{*}{ \scalebox{0.9}[1]{\textbf{MLTL2p}}}     &   \multirow{2}{*}{\scalebox{0.8}[1]{\textbf{LRTFL0}}}      &   \multirow{2}{*}{\scalebox{0.8}[1]{\textbf{SNLRSF}}}    &   \multirow{2}{*}{\scalebox{0.8}[1]{\textbf{LRTD}}}    &  \scalebox{0.8}[1]{\textbf{FGLR}}   &   \multirow{2}{*}{\scalebox{0.8}[1]{\textbf{QRNN3D}}}\\ 
 &   &  &   &   &   &   &  \scalebox{0.8}[1]{\textbf{+BM4D}}   &  \\ \hline 
\multirow{4}{*}{1} & MPSNR        & 19.63      & \textbf{39.35} & 37.02  & 39.10  & 30.97  & 33.65  & 31.56         \\
                   & MSSIM        & 0.276      & \textbf{0.969} & 0.940  & 0.965  & 0.807  & 0.900 & 0.856         \\
                   & MFSIM        & 0.613      & \textbf{0.981} & 0.968  & 0.978  & 0.907  & 0.946  & 0.926         \\
                   %& MSAM         & 472.9      & \textbf{44.20}  & 59.70   & 46.40   & 135.0  & 116.0  & 109.8         \\
                   & ERGAS        & 0.426      & \textbf{0.041} & 0.057  & 0.043  & 0.114  & 0.111  & 0.110         \\ \hline
\multirow{4}{*}{2} & MPSNR        & 18.52      & \textbf{39.03} & 35.99  & 34.97  & 30.64  & 32.77  & 31.05         \\
                   & MSSIM        & 0.245      & \textbf{0.967} & 0.929  & 0.915  & 0.797  & 0.887  & 0.840         \\
                   & MFSIM        & 0.580      & \textbf{0.980} & 0.962  & 0.953  & 0.902  & 0.940  & 0.918         \\
                   %& MSAM         & 538.0      & \textbf{46.00}  & 65.80   & 76.60   & 144.5  & 229.4  & 116.7         \\
                   & ERGAS        & 0.470      & \textbf{0.043} & 0.064  & 0.072  & 0.120  & 0.117  & 0.118         \\ \hline
\multirow{4}{*}{3} & MPSNR        & 14.46      & \textbf{35.51} & 34.91  & 26.24  & 29.32  & 32.07  & 30.88         \\
                   & MSSIM        & 0.234      & \textbf{0.929} & 0.920  & 0.704  & 0.782  & 0.874  & 0.829         \\
                   & MFSIM        & 0.572      & \textbf{0.967} & 0.957  & 0.856  & 0.903  & 0.934  & 0.916         \\
                   %& MSAM         & 941.5      & \textbf{119.7} & 157.5  & 280.0  & 193.5  & 861.4  & 131.1         \\
                   & ERGAS        & 0.600      & \textbf{0.074} & 0.086  & 0.170  & 0.133  & 0.125  & 0.119         \\ \hline
\multirow{4}{*}{4} & MPSNR        & 16.10      & \textbf{36.49} & 34.64  & 35.33  & 26.47  & 29.12  & 30.35         \\
                   & MSSIM        & 0.176      & \textbf{0.945} & 0.909  & 0.925  & 0.579  & 0.793  & 0.819         \\
                   & MFSIM        & 0.513      & \textbf{0.966} & 0.949  & 0.955  & 0.804  & 0.900  & 0.906         \\
                   %& MSAM         & 766.3      & \textbf{70.90}  & 126.5  & 98.70   & 226.3  & 394.0  & 125.5         \\
                   & ERGAS        & 0.593      & \textbf{0.061} & 0.099  & 0.078  & 0.208  & 0.146  & 0.125         \\ \hline
& Mtime & $-$ & 299.0         & 265.9 & 302.2 & 207.7 & 107.5 & \textbf{0.6} \\ \hline
\end{tabular}
\end{table}
The average numerical results for each case over all the simulated HSIs are presented in Table~\ref{table1}. The results in bold font indicate the best performance of the evaluation metric in the current case. It can be observed that the proposed MLTL2p method outperforms other methods in terms of all the evaluation metrics. For example, in case 2,  the MPSNR value of the HSI restored by the MLTL2p method is 3~dB larger than the MPSNR value of the second best method, that is, the LRTFL0 method.

For visual quality comparison, the HSIs restored by different methods are presented in Fig. \ref{fig:Noisy_Chikusei_CASE_1}-\ref{fig:Noisy_XiongAn_CASE_4}, with a subregion marked by a white box and enlarged in a white box. The MLTL2p method achieves state-of-the-art performance for HSI denoising and destriping, while the competing methods seem to fail to remove noise or restore HSIs with high quality. For example, for removing mixed noise, the BM4D method fails to remove stripes even if the HSI only contains few stripes as shown in Fig.~\ref{fig:Noisy_Chikusei_CASE_1}(\subref{subfig:BM4D});  the SNLRSF and QRNN3D methods may not remove the stripes and dead lines, when the noise percentage is high, for example, in Fig.~\ref{fig:Noisy_DC_CASE_3}(\subref{subfig:SNLRSF1})(\subref{subfig:QRNN3D1}). Also, regarding the quality of the restoration, the HSIs restored by the LRTFL0 method may contain staircase artifacts, e.g., in Fig. \ref{fig:Noisy_Chikusei_CASE_1}(\subref{subfig:LRTFL01}); the HSIs restored by the LRTD method contain blur and artifacts, e.g., in Fig.~\ref{fig:Noisy_Chikusei_CASE_1}(\subref{subfig:LRTD}); the HSIs restored by the QRNN3D method look noisy, e.g., in Fig.~\ref{fig:Noisy_Chikusei_CASE_1}(\subref{subfig:QRNN3D1}).
\begin{figure}[!h]
	\centering
	\begin{subfigure}[h]{0.242\linewidth}
    \includegraphics[width=\linewidth]{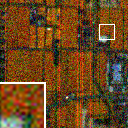}
    \caption{Noisy image}
    \end{subfigure}	
    \begin{subfigure}[h]{0.242\linewidth}
    \includegraphics[width=\linewidth]{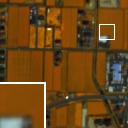}
    \caption{Ground truth}
    \end{subfigure}
    \begin{subfigure}[h]{0.242\linewidth}
    \includegraphics[width=\linewidth]{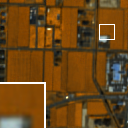}
    \caption{MLTL2p (ours)}
    \end{subfigure}
     \begin{subfigure}[h]{0.242\linewidth}
    \includegraphics[width=\linewidth]{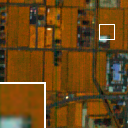}
    \caption{LRTFL0}\label{subfig:LRTFL01}
    \end{subfigure}
    \begin{subfigure}[h]{0.242\linewidth}
    \includegraphics[width=\linewidth]{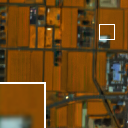}
	\caption{SNLRSF}\label{subfig:SNLRSF1}
    \end{subfigure}
    \begin{subfigure}[h]{0.242\linewidth}
    \includegraphics[width=\linewidth]{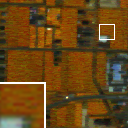}
	\caption{LRTD}\label{subfig:LRTD}
    \end{subfigure}
    \begin{subfigure}[h]{0.242\linewidth}
    \includegraphics[width=\linewidth]{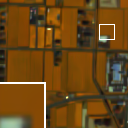}
	\caption{FGLR+BM4D}
  \end{subfigure}
    \begin{subfigure}[h]{0.242\linewidth}
    \includegraphics[width=\linewidth]{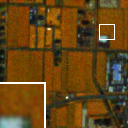}
    \caption{QRNN3D}\label{subfig:QRNN3D1}
    \end{subfigure}
	\caption{Comparison of HSIs (R:112, G:70, B:20) restored by different methods from Chikusei in case 1. The PSNR value for each restored HSI: (a) Noisy image (19.63~dB); (c) MLTL2p (ours) (39.57~dB); (d) LRTFL0 (37.05~dB); (e) SNLRSF (38.70~dB); (f) LRTD (31.82~dB); (g) FGLR+BM4D (33.97~dB);  (h) QRNN3D (32.20~dB).}\label{fig:Noisy_Chikusei_CASE_1} 
\end{figure}
\begin{figure}[!h]
	\centering
	\begin{subfigure}[h]{.242\linewidth}
    \includegraphics[width=\linewidth]{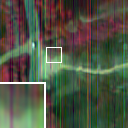}
    \caption{Noisy image}
    \end{subfigure}	
    \begin{subfigure}[h]{.242\linewidth}
    \includegraphics[width=\linewidth]{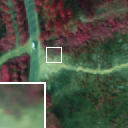}
    \caption{Ground truth}
    \end{subfigure}
    \begin{subfigure}[h]{.242\linewidth}
    \includegraphics[width=\linewidth]{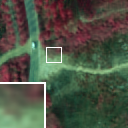}
    \caption{MLTL2p (ours)}
    \end{subfigure}
    \begin{subfigure}[h]{.242\linewidth}
    \includegraphics[width=\linewidth]{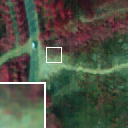}
    \caption{LRTFL0}
    \end{subfigure}
    \begin{subfigure}[h]{.242\linewidth}
    \includegraphics[width=\linewidth]{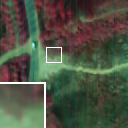}
	\caption{SNLRSF}
    \end{subfigure}
    \begin{subfigure}[h]{.242\linewidth}
    \includegraphics[width=\linewidth]{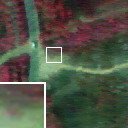}
	\caption{LRTD}
    \end{subfigure}
    \begin{subfigure}[h]{.242\linewidth}
    \includegraphics[width=\linewidth]{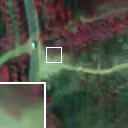}
	\caption{FGLR+BM4D}\label{subfig:BM4D}
  \end{subfigure}
    \begin{subfigure}[h]{.242\linewidth}
    \includegraphics[width=\linewidth]{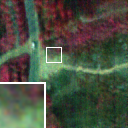}
    \caption{QRNN3D}
    \end{subfigure}
	\caption{Comparison of HSIs (R:118, G:78, B:38) restored by different methods from Xiong-An in case 2. The PSNR value for each restored HSI: (a) Noisy image (18.52~dB); (c) MLTL2p (ours) (38.82~dB); (d) LRTFL0 (35.74~dB); (e) SNLRSF (34.76~dB); (f) LRTD (31.90~dB); (g) FGLR+BM4D (34.95~dB);  (h) QRNN3D (31.01~dB).}\label{fig:Noisy_Xiong-An_CASE_2} 
\end{figure}
\begin{figure}[!h]
	\centering
	\begin{subfigure}[h]{.242\linewidth}
    \includegraphics[width=\linewidth]{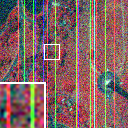}
    \caption{Noisy image}
    \end{subfigure}	
    \begin{subfigure}[h]{.242\linewidth}
    \includegraphics[width=\linewidth]{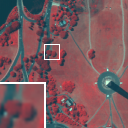}
    \caption{Ground truth}
    \end{subfigure}
    \begin{subfigure}[h]{.242\linewidth}
    \includegraphics[width=\linewidth]{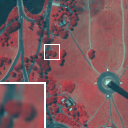}
    \caption{MLTL2p (ours)}
    \end{subfigure}
    \begin{subfigure}[h]{.242\linewidth}
    \includegraphics[width=\linewidth]{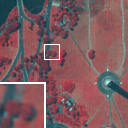}
    \caption{LRTFL0}
    \end{subfigure}
    \begin{subfigure}[h]{.242\linewidth}
    \includegraphics[width=\linewidth]{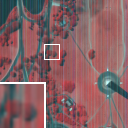}
	\caption{SNLRSF}\label{subfig:SNLRSF2}
    \end{subfigure}
    \begin{subfigure}[h]{.242\linewidth}
    \includegraphics[width=\linewidth]{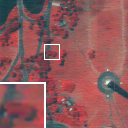}
	\caption{LRTD}
    \end{subfigure}
    \begin{subfigure}[h]{.242\linewidth}
    \includegraphics[width=\linewidth]{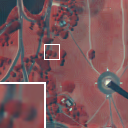}
	\caption{FGLR+BM4D}
  \end{subfigure}
    \begin{subfigure}[h]{.242\linewidth}
    \includegraphics[width=\linewidth]{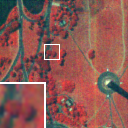}
    \caption{QRNN3D}
    \end{subfigure}
	\caption{Comparison of HSIs (R:60, G:27, B:17) restored by different methods from Washington DC Mall in case 3. The PSNR value for each restored HSI: (a) Noisy image (14.66~dB); (c) MLTL2p (ours) (35.28~dB); (d) LRTFL0 (34.62~dB); (e) SNLRSF (26.65~dB); (f) LRTD (28.52~dB); (g) FGLR+BM4D (30.69~dB); (h) QRNN3D (30.01~dB).}\label{fig:Noisy_DC_CASE_3}
\end{figure}
\begin{figure}[!h]
	\centering
	\begin{subfigure}[h]{.242\linewidth}
	\includegraphics[width=\linewidth]{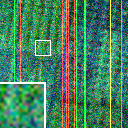}
	\caption{Noisy image}
\end{subfigure}	
\begin{subfigure}[h]{.242\linewidth}
	\includegraphics[width=\linewidth]{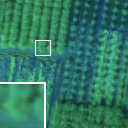}
	\caption{Ground truth}
\end{subfigure}
\begin{subfigure}[h]{.242\linewidth}
	\includegraphics[width=\linewidth]{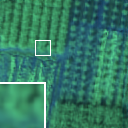}
	\caption{MLTL2p (ours)}
\end{subfigure}
\begin{subfigure}[h]{.242\linewidth}
	\includegraphics[width=\linewidth]{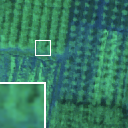}
	\caption{LRTFL0}\label{subfig:LRTFL02}
\end{subfigure}
\begin{subfigure}[h]{.242\linewidth}
	\includegraphics[width=\linewidth]{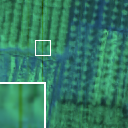}
	\caption{SNLRSF}
\end{subfigure}
\begin{subfigure}[h]{.242\linewidth}
	\includegraphics[width=\linewidth]{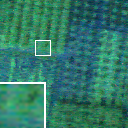}
	\caption{LRTD}
\end{subfigure}
\begin{subfigure}[h]{.242\linewidth}
	\includegraphics[width=\linewidth]{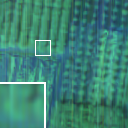}
	\caption{FGLR+BM4D}
\end{subfigure}
\begin{subfigure}[h]{.242\linewidth}
	\includegraphics[width=\linewidth]{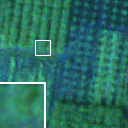}
	\caption{QRNN3D}\label{subfig:QRNN3D2}
\end{subfigure}
	\caption{Comparison of HSIs (R:107, G:73, B:21) restored by different methods from Xiong-An  in case 4. The PSNR value for each restored HSI: (a) Noisy image (16.02~dB); (c) MLTL2p (ours) (36.66~dB); (d) LRTFL0 (34.95~dB); (e) SNLRSF (35.46~dB); (f) LRTD (25.91~dB); (g) FGLR+BM4D (30.11~dB); (h) QRNN3D (29.54~dB).}\label{fig:Noisy_XiongAn_CASE_4}
\end{figure}

\newpage\subsection{Real data experiments} In this subsection, we test the proposed method and the competing methods on two real HSI datasets containing mixed noise. The test images are subimages of size $128 \times 128 \times 128$
randomly obtained from the HYDICE Urban\footnote{\href{ http://www.erdc.usace.army.mil/Media/Fact-Sheets/Fact-Sheet-Article-View/Article/610433/hypercube/}{http://www.erdc.usace.army.mil/Media/Fact-Sheets/Fact-Sheet-Article-View/Article/610433/hypercube/}}~$(307\times307\times 210)$ and EO-1 Hyperion\footnote{\href{http://www.lmars.whu.edu.cn/prof_web/zhanghongyan/resource/noise_EOI.zip}{http://www.lmars.whu.edu.cn/prof\_web/zhanghongyan/resource/noise\_EOI.zip}}~$(400 \times 200 \times 166)$. A selected band of the HSI restored by each method is presented in Fig.~\ref{fig:URBAN} and Fig.~\ref{fig:EO-1 Hyperion Dataset} for HYDICE Urban dataset and EO-1 Hyperion dataset, respectively.
It can be observed that the proposed MLTL2p method can remove the stripes while preserving the image details. However,  the
LRTFL0, LRTD, BM4D and QRNN3D methods are unable to eliminate the stripes when the band is contaminated by heavy mixed noise   as shown in Fig.~\ref{fig:URBAN}; and the SNLRSF, LRTD, BM4D and QRNN3D methods remove not only the noise but also some structural details of the HSI as shown in Fig.~\ref{fig:EO-1 Hyperion Dataset}.

\begin{figure}[h]
	\centering
	\begin{subfigure}[t]{.242\linewidth}
\includegraphics[width=\linewidth]{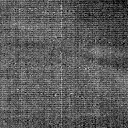}
    \caption{Noisy image}
    \end{subfigure}	
    \begin{subfigure}[t]{.242\linewidth}
     \includegraphics[width=\linewidth]{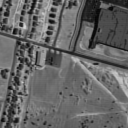}
    \caption{MLTL2p (ours)}
    \end{subfigure}
    \begin{subfigure}[t]{.242\linewidth}
    \includegraphics[width=\linewidth]{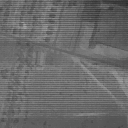}
    \caption{LRTFL0}
    \end{subfigure}
    \begin{subfigure}[t]{.242\linewidth}
    \includegraphics[width=\linewidth]{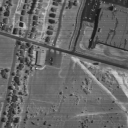}
	\caption{SNLRSF}
    \end{subfigure}
    \begin{subfigure}[t]{.242\linewidth}
    \includegraphics[width=\linewidth]{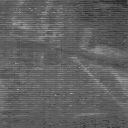}
	\caption{LRTD}
    \end{subfigure}
    \begin{subfigure}[t]{.242\linewidth}
    \includegraphics[width=\linewidth]{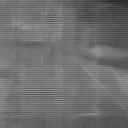}
	\caption{FGLR+BM4D}
  \end{subfigure}
    \begin{subfigure}[t]{.242\linewidth}
    \includegraphics[width=\linewidth]{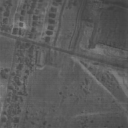}
    \caption{QRNN3D}
    \end{subfigure}
	\caption{Comparison of the 25-th band of the HSI restored by different methods from HYDICE Urban.}\label{fig:URBAN}
\end{figure}
\begin{figure}[htbp]
	\centering
	\begin{subfigure}[t]{.242\linewidth}
\includegraphics[width=\linewidth]{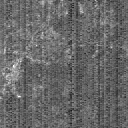}
    \caption{Noisy image}
    \end{subfigure}	
    \begin{subfigure}[t]{.242\linewidth}
    \includegraphics[width=\linewidth]{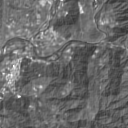}
    \caption{MLTL2p (ours)}
    \end{subfigure}
    \begin{subfigure}[t]{.242\linewidth}
	\includegraphics[width=\linewidth]{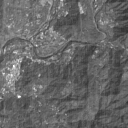}
	\caption{LRTFL0}
	\end{subfigure}
    \begin{subfigure}[t]{.242\linewidth}
    \includegraphics[width=\linewidth]{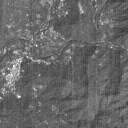}
	\caption{SNLRSF}
    \end{subfigure}
    \begin{subfigure}[t]{.242\linewidth}
    \includegraphics[width=\linewidth]{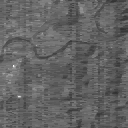}
	\caption{LRTD}
    \end{subfigure}
    \begin{subfigure}[t]{.242\linewidth}
    \includegraphics[width=\linewidth]{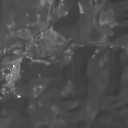}
	\caption{FGLR+BM4D}
  \end{subfigure}
    \begin{subfigure}[t]{.242\linewidth}
    \includegraphics[width=\linewidth]{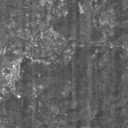}
    \caption{QRNN3D}
    \end{subfigure}
      	
	\caption{Comparison of the 128-th band of the HSI restored by different methods from EO-1 Hyperion.}\label{fig:EO-1 Hyperion Dataset}
\end{figure}

\subsection{Numerical Convergence Analysis}

We conduct numerical convergence analysis on the MLTL2p method in the second phase. In particular, we analyze the convergence of the sequence $\{(\mathcal{S}^k,[\boldsymbol{X}_i^k]_g,[\boldsymbol{X}_i^k]_l,[\boldsymbol{X}_{i}^k]_{nl},[\mathcal{G}^k]_g,[\mathcal{G}^k]_l,[\mathcal{G}^k]_{nl},\mathcal{L}^k)\}$ generated by the P-BCD algorithm in Algorithm~\ref{alg2} for model~\eqref{model:HSI_new}.
	
First, we analyze the convergence of $\mathcal{L}^k$ and three terms $\mathcal{L}_{g}^k$, $\mathcal{L}_{l}^k$ and $\mathcal{L}_{nl}^k$, associated with $[\boldsymbol{X}_i^k]_g,[\boldsymbol{X}_i^k]_l,[\boldsymbol{X}_{i}^k]_{nl},[\mathcal{G}^k]_g,[\mathcal{G}^k]_l$, and $[\mathcal{G}^k]_{nl}$. The terms $\mathcal{L}_{g}^k$, $\mathcal{L}_{l}^k$ and $\mathcal{L}_{nl}^k$ are defined as follows
\begin{align*}
	\mathcal{L}_{g}^k=&\mathcal{W}_g^{-1}\mathscr{R}_g^{\top}([\mathcal{Y}^{k}]_g),\\
	\mathcal{L}_{l}^k=&\mathcal{W}_l^{-1}\mathscr{R}_l^{\top}([\mathcal{Y}^{k}]_l),\\
	\mathcal{L}_{nl}^k=&\mathcal{W}_{nl}^{-1}\mathscr{R}_{nl}^{\top}([\mathcal{Y}^{k}]_l),
\end{align*}
where
	\begin{align*}
	&[\mathcal{Y}^{k}]_g= [\mathcal{G}^k]_g\times
	_{1}[\boldsymbol{X}_1^k]_g\times_{2}[\boldsymbol{X}_2^k]_g\times_{3}[\boldsymbol{X}_3^k]_g,\\
	&[\mathcal{Y}^{k}]_l= [\mathcal{G}^k]_l\times
	_{1}[\boldsymbol{X}_1^k]_l\times_{2}[\boldsymbol{X}_2^k]_l\times_{3}[\boldsymbol{X}_3^k]_l,\\
	&[\mathcal{Y}^{k}]_{nl}= [\mathcal{G}^k]_{nl}\times
	_{1}[\boldsymbol{X}_1^k]_{nl}\times_{2}[\boldsymbol{X}_2^k]_{nl}\times_{3}[\boldsymbol{X}_3^k]_{nl}.
\end{align*}
Note that $\mathcal{L}_{g}^k$, $\mathcal{L}_{l}^k$ and $\mathcal{L}_{nl}^k$ represent the restored images obtained from global, local, nonlocal low-rank regularization, respectively, while $\mathcal{L}^k$ represents the overall restored image, according to line 9 in Algorithm~\ref{alg2}. To analyze the numerical convergence, we compute the relative errors of $\mathcal{L}^k$, $\mathcal{L}_{g}^k$, $\mathcal{L}_{l}^k$, and $\mathcal{L}_{nl}^k$, e.g., $$\frac{\|\mathcal{L}^k-\mathcal{L}^{k-1}\|_F}{\|\mathcal{L}^k\|_F}$$ for $\mathcal{L}^k$. In Fig.~\ref{fig:conv}(\subref{fig:conv_L}), we plot the relative errors vs the number of iterations and observe that these restored images exhibit similar convergence behaviors.

Second, we  analyze the convergence of $\mathcal{S}^k$ using both the relative error and the group sparsity rate.  In Fig.~\ref{fig:conv}(\subref{fig:conv_S}), we plot the relative errors vs the number of iterations on the left axis, and the group sparsity rate, defined as the proportion of nonzero mode-$1$ fibers of $\mathcal{S}^k$, on the right axis. We observe that  $\mathcal{S}^k$ converges stably in terms of both relative error and group sparsity rate.

\begin{figure}[h]
	\centering
	\begin{subfigure}[t]{.48\linewidth}
		\includegraphics[width=\linewidth]{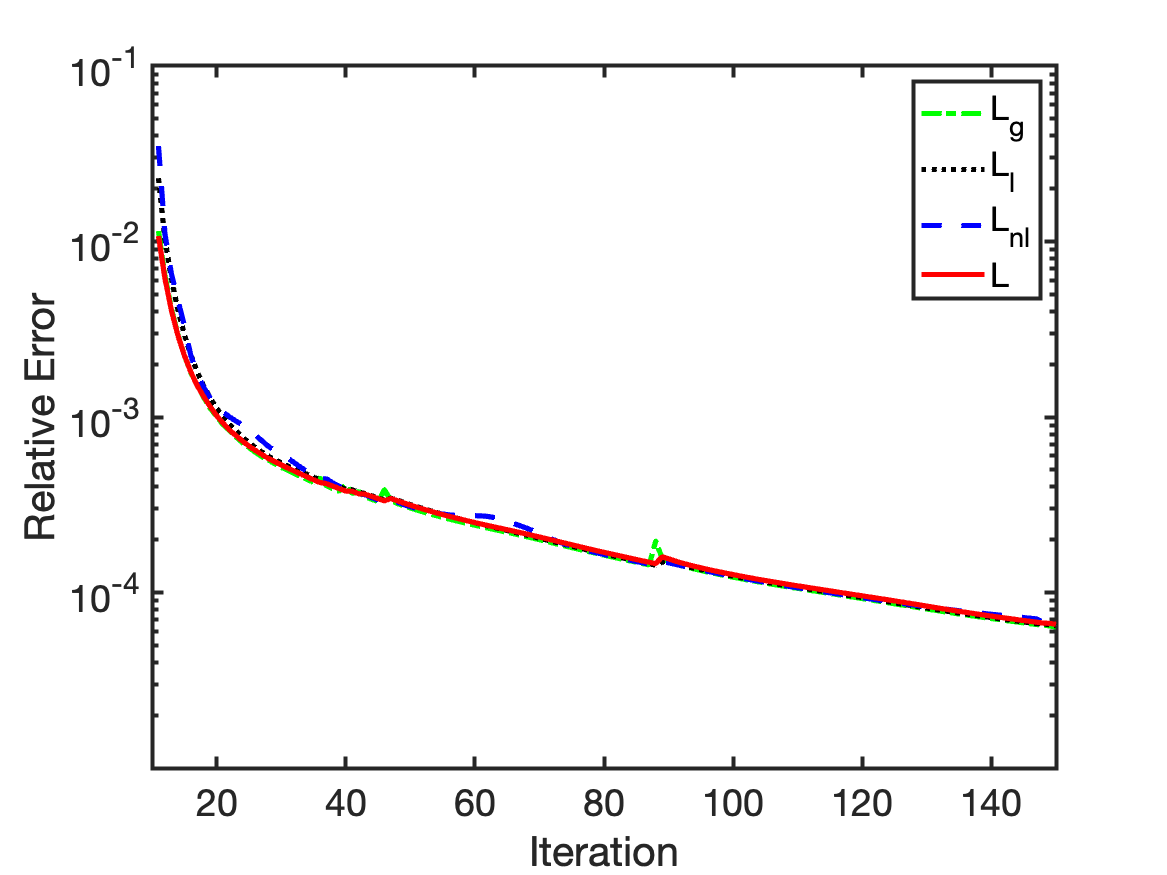}
		\caption{$\mathcal{L}_{g}$, $\mathcal{L}_{l}$, $\mathcal{L}_{nl}$, and $\mathcal{L}$}\label{fig:conv_L}
	\end{subfigure}
	\begin{subfigure}[t]{.48\linewidth}
		\includegraphics[width=\linewidth]{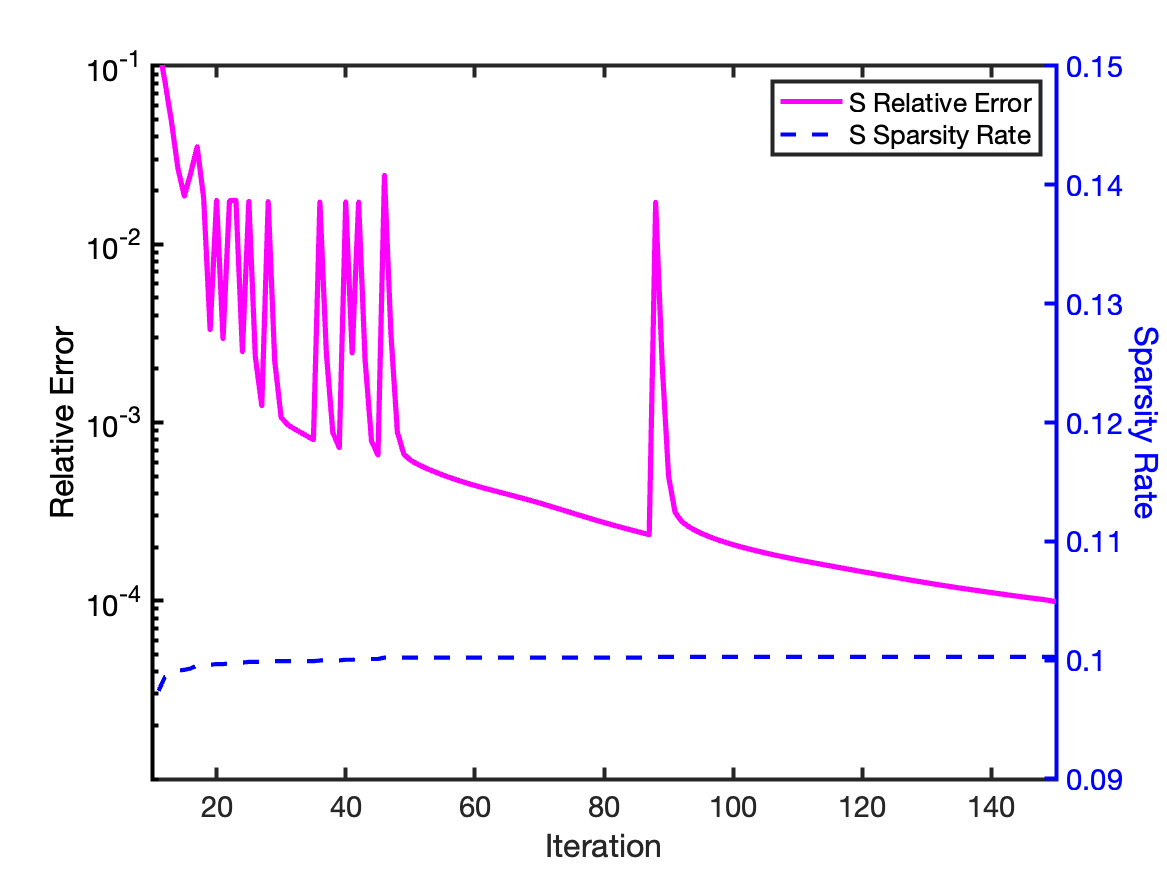}
		\caption{$\mathcal{S}$}\label{fig:conv_S}
	\end{subfigure}
	\caption{Numerical convergence analysis.}\label{fig:conv}
\end{figure}

\subsection{Ablation Experiments}

We conduct ablation experiments to demonstrate the advantages of our proposed two-phase MLTL2p method, which leverages multi-scale low-rankness, over the methods that utilize only one or two scales of low-rankness. We use Case 2 of Xiong-An for testing.  Let G denote global low-rankness, L denote local block low-rankness, and NL denote nonlocal self-similar tensor low-rankness. We compare our two-phase MLTL2p method, also referred to as the two-phase G+L+NL method, with five other methods: single-phase G, L, and G+L methods, as well as two-phase G+NL and L+NL methods.   In Fig.~\ref{fig:ablation}, we plot the PSNR values of the restored images obtained by each method vs the number of iterations. The result shows that two-phase methods outperform single-phase methods, and the multi-scale method outperforms those using only one or two scales.
\begin{figure}[h]
	\centering
	\includegraphics[width=.62\linewidth]{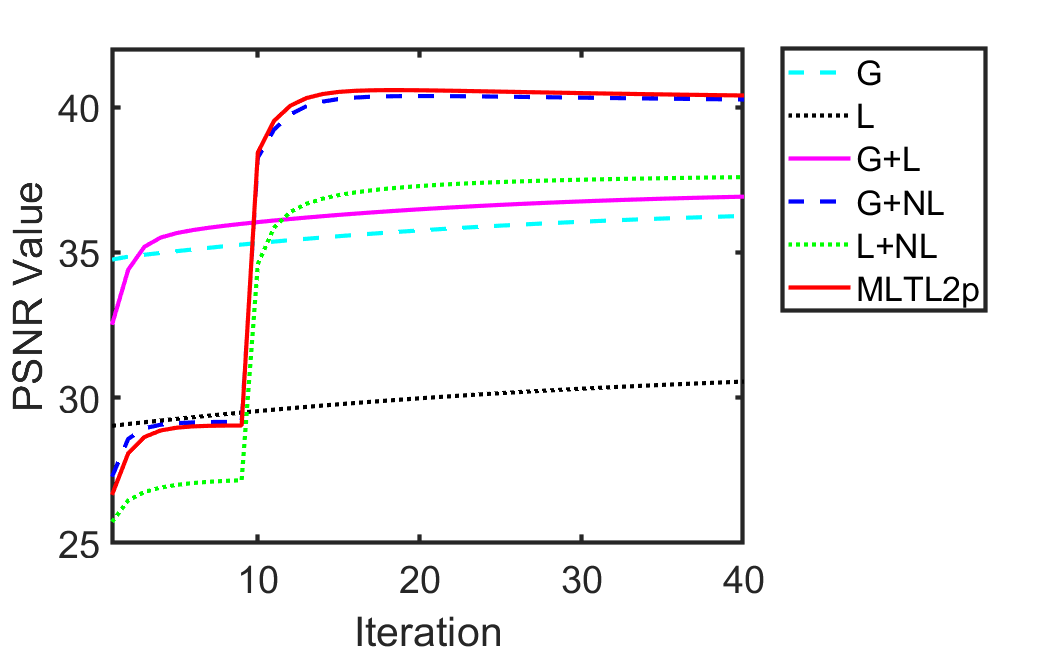}
	\caption{Ablation experiments.}\label{fig:ablation}
\end{figure}

\subsection{Parameter Analysis}

We conduct a parameter analysis on several model parameters. We use Case 3 of Washington DC Mall for testing the MLTL2p method in the second phase. First, we test parameters $\gamma$ and $p$ associated with the variable $\mathcal{S}$.  The plot of PSNR values vs parameters  $\gamma$ and $p$ is presented in Fig.~\ref{fig:parameter}(\subref{fig:para_S}). The PSNR value achieves the best performance when $\gamma\in [2.2,3.5]$ and $p\in[0.1,0.2]$. Hence, we set $\gamma=2.2$ and $p=0.1$.

Second, we test parameters $w_j$ and $n_3$ associated with the variable $[\mathcal{G}]_g,[\mathcal{G}]_l$, and $[\mathcal{G}]_{nl}$. Note that in the second phase of the proposed method we set $w_j=10^{-2}$ and $n_3=5, 3,$ and $5$, respectively, for $[\mathcal{G}]_g,[\mathcal{G}]_l$, and $[\mathcal{G}]_{nl}$. Fig.~\ref{fig:parameter}(\subref{fig:para_G}) represents the plot of PSNR values vs parameters $w_j$ and $n_3$ offset, where $n_3$ offset is the deviation from the selected $n_3$ values. The PSNR values remain high when the $n_3$ offset is close to zero and  $w_j\geq 10^{-2}$.
\begin{figure}[h]
	\centering
	\begin{subfigure}[t]{.44\linewidth}
		\includegraphics[height=0.75\linewidth]{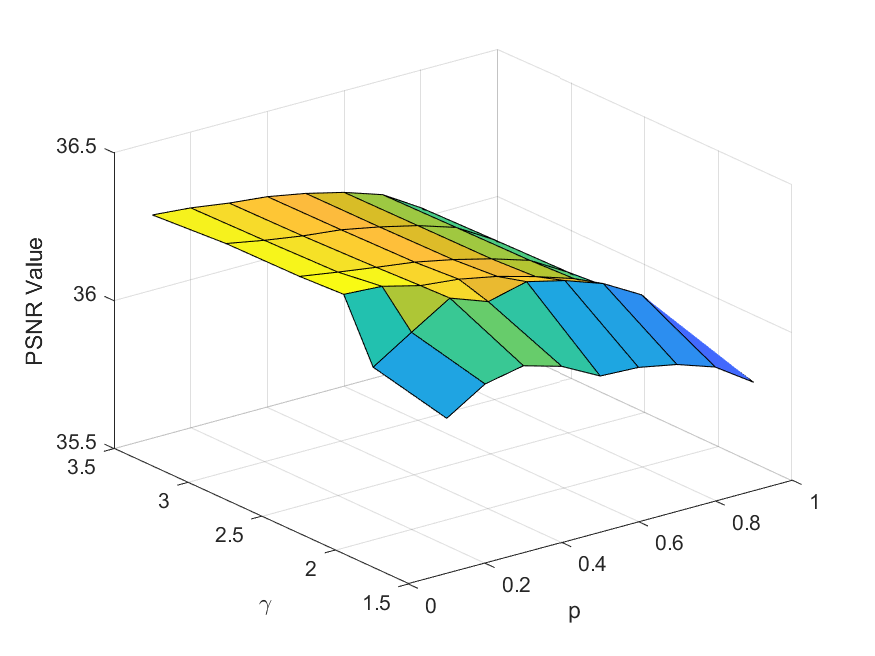}
		\caption{Parameters  $\gamma$ and $p$}\label{fig:para_S}
	\end{subfigure}
	\begin{subfigure}[t]{.44\linewidth}
		\includegraphics[height=0.75\linewidth]{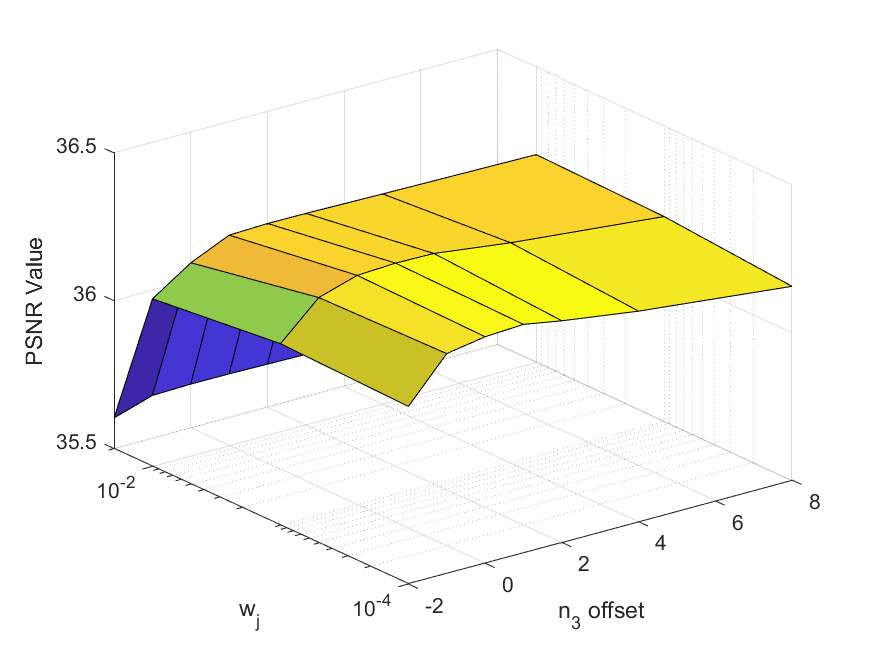}
		\caption{Parameters $w_j$ and $n_3$ offset}\label{fig:para_G}
	\end{subfigure}
	\begin{subfigure}[t]{.1\linewidth}
		\includegraphics[height=3.3\linewidth]{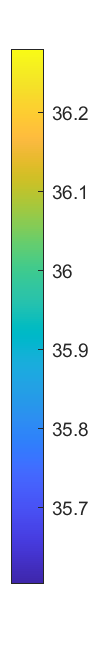}
	\end{subfigure}
	\caption{Parameter analysis.}\label{fig:parameter}
\end{figure}

% % -------------------- 6. Conclusion Section -------------------- % %

\section{Conclusions}\label{sec:conclusions}
In this paper, we propose an HSI denoising and destriping method, which has an optimization model given in~\eqref{model:HSI} and an iterative algorithm described in Algorithm~\ref{alg1}. The optimization model consists of a data fidelity term, a sparsity-enhanced nonlocal low-rank tensor regularization term for denoising, and a $\ell_{2,p}$ norm for destriping. The iterative algorithm is proposed using a proximal version of BCD algorithms, which has convergence guarantees. The proposed model can be extended to a multi-scale low-rank regularized model and the numerical experiments tested on simulated and real HSIs also show the effectiveness of our proposed method in removing Gaussian noise, stripes, and deadlines.

\section*{Acknowledgments}
The authors would like to thank the anonymous reviewers and the associate editor for their valuable comments and suggestions, which led to significant improvements in this article.

% % -------------------- References Section -------------------- % %

\bibliographystyle{siam}
\bibliography{bibfile}

\end{document}